\documentclass[11pt,english]{article}

\usepackage{ulem}
\usepackage{mathrsfs}
\usepackage[latin1]{inputenc}
\usepackage{babel}
\usepackage{amsthm}
\usepackage{epsfig}
\usepackage
{graphicx}
\usepackage{color,amsmath,amssymb}
\newcommand{\As}{\tilde A_{11}^s}
\newcommand{\sym}{\mathrm{Sym}^-_J}

\newcommand{\mau}{\Theta}

\newcommand{\mad}{\Upsilon}

\newcommand{\lau}{\kappa}
\newcommand{\lad}{\cak}

\newcommand{\barm}{\bar{m}}

\newcommand{\F}{\mathscr{F}}
\newcommand{\M}{\mathscr{M}}
\renewcommand{\L}{\mathscr{L}}

\newcommand{\GL}{GL}

\newcommand{\LL}{\Lambda}
\newcommand{\ooverrightarrow}[1]{\LL(#1)}

\newcommand{\ps}[2]{\la #1, #2\ra}

\DeclareMathOperator{\diag}{diag}

\DeclareMathOperator{\SO}{SO}

\DeclareMathOperator{\rank}{rank}

\newcommand{\T}[1]{{#1}^\mathrm{T}}

\newcommand{\rk}{\rank}
\newcommand{\spann}{\mathrm{span}}

\newcommand{\add}[1]{\textcolor{black}{#1}}

\makeatletter
\newcommand{\@syst}[1]{%
\[ 
\left\{
\begin{aligned}
#1
\end{aligned}
\right.
\]
}
\newenvironment{syst}{\collect@body\@syst}{\global\@ignoretrue}

\newcommand{\@systn}[1]{%
\begin{equation}
\left\{
\begin{aligned}
#1
\end{aligned}
\right.
\end{equation}
}

\makeatother

\newtheorem{thm}{Theorem}[section]
\newtheorem{lemma}[thm]{Lemma}
\newtheorem{cor}[thm]{Corollary}
\newtheorem{prop}[thm]{Proposition}
\newtheorem{rmk}[thm]{Remark}

\newtheorem{deff}[thm]{Definition}

\newcommand{\beq}{\begin{equation}} 
\newcommand{\eeq}{\end{equation}}
\newcommand{\be}{\begin{equation}} 
\newcommand{\ee}{\end{equation}}
\newcommand{\bea}{\begin{eqnarray}} 
\newcommand{\eea}{\end{eqnarray}}
\newcommand{\brs}{\begin{eqnarray*}}
\newcommand{\ers}{\end{eqnarray*}}
\newcommand{\ba}{\begin{array}} 
\newcommand{\ea}{\end{array}}
\newcommand{\br}{\begin{eqnarray}}
\newcommand{\er}{\end{eqnarray}}

\newcommand{\lp}{\left(}
\newcommand{\rp}{\right)}
\newcommand{\la}{\left\langle}
\newcommand{\ra}{\right\rangle}

\renewcommand{\r}[1]{(\ref{#1})}

\def\EOP{\ \hfill $\Box$}
\newcommand{\prf}{{\it Proof.}}

\newcommand{\R}{\mathbf{R}}
\newcommand{\RP}{\R\mathbf{P}}
\newcommand{\N}{\mathbf{N}}

\newcommand{\Id}{\mathrm{Id}}
\newcommand{\Int}{\mathrm{Int}}
\newcommand{\Cl}{\mathrm{Clos}}

\def\mt{\mapsto}

\newcommand{\om}{\omega} 
\newcommand{\Om}{\Omega}

\newcommand{\lb}{\lambda}

\def\eps{\varepsilon}
\def\fhi{\varphi}
\def\al{\alpha}
\def\th{\theta}

\newcommand{\caa}{{\cal {A}}}
\newcommand{\caA}{{\cal {A}}}\newcommand{\A}{{\cal {A}}}

\newcommand{\caS}{{\cal {S}}}

\newcommand{\cak}{{\cal {K}}}

\newcommand{\CC}{{\cal {C}}}

\newcommand{\0}{
0}

\title{\bf \Large 
Controllability properties of a class of systems modeling 
 swimming microscopic organisms
}
\author{        Mario Sigalotti and Jean-Claude Vivalda
}
\date{}
\begin{document}

\maketitle

\today

\begin{abstract}
We consider a finite-dimensional model for the motion of   
microscopic organisms whose propulsion 
exploits
the action of a layer of {\it cilia} covering its surface.  
The model couples 
Newton's laws driving the organism, 
considered as 
a rigid body, with
Stokes equations governing the surrounding fluid.
The action of the
{\it cilia} is described by a set of controlled 
velocity fields  on the surface of the organism. 
The first contribution of the paper is the proof 
that  such a system
is generically controllable 
when the space of controlled velocity fields is at least three-dimensional.  We also provide 
a complete characterization of controllable systems 
in the case in which
the organism has a spherical shape. Finally, we 
offer  a complete picture of controllable and non-controllable systems under the additional hypothesis 
that
the organism and the fluid have 
densities  of the same order of magnitude. 
\end{abstract}

\section{Introduction}\label{intro}
The mathematical 
description of the motion of swimming organisms 
and of the mechanisms of their propulsion is a 
challenging and 
wide field of research. 
This paper aims at giving a contribution 
to such a domain in a comparatively narrow direction. 

The object of our study are microscopic organisms.
The most prominent aspect 
of the models describing 
their motion 
is the very high level of viscosity (\cite{berg,purcell,taylor}). 
The mechanism of their 
propulsion
depends heavily on the different species and 
admits a great variety of models (see, for instance,
\cite{desimone,childress,galdi,lighthill}
and references therein).

Among microscopic organisms we will further restrict our attention to
the class of {\it ciliata}, whose propulsion is determined
by {\it cilia},
hair-like organelles covering their surface 
and whose size is very small compared with that of the organism
(see \cite{blake,brennen,galdi}).
Recently San Mart\'\i n, Takahashi and Tucsnak \cite{nancy}
 proposed 
a control theory approach to the description of 
ciliata. Following the former literature (and in particular \cite{galdi,keller}) they assume that the organism is 
a rigid body and that the motion of the
cilia is described by a set of controlled 
velocity fields  on a surface enclosing the layer of cilia. 
They are therefore faced to an infinite-dimensional  control system
coupling 
Navier--Stokes equations, which describe the motion of 
the fluid, and the equations for the 
rigid body arising form Newton's laws. 
The hypothesis of high viscosity (ie, of low Raynolds number) mentioned above 
is exploited to reduce such a control system to a finite-dimensional version of it. The state-space is given by the coordinates of the center of mass of the organism, its orientation and its angular and linear velocities. 

Summarizing, San Mart\'\i n, Takahashi and Tucsnak single out a $12$-dimensional nonlinear control system which is affine in the controls. The control system depends on a number of parameters
which are explicitly derived as functions of the shape of the organism, of its mass distribution, and of the 
surface 
velocity fields describing the 
motion  of the cilia. The authors study the control properties of 
the system proving that 
it is generically controllable when the number of 
controlled vector fields is at least six (see Section~\ref{gen} for details). They also provide some example of controllable 
systems in the case in which the organism has a spherical shape. In this special case they also 
study the controllability properties 
of the linearization of the system at its rest position.

Our scope is to 
enhance 
the results of \cite{nancy} basically in three directions: firstly we want to give sharper results on the generic controllability of the system; secondly we want to obtain a complete characterization of controllable systems 
in the case in which the organism is spherical; thirdly we want to 
provide a complete picture of controllable and non-controllable systems in the simplified  situation which --following again \cite{nancy}-- 
comes into effect in the case in which the organism and the fluid have 
densities  of the same order of magnitude. 
The overall aim is to provide a better understanding of
the dependence of the controllability properties of the model proposed in \cite{nancy} on the physical quantities characterizing it. 
The 
hard aspect
of the model is indeed 
the 
difficulty
to obtain 
reliable expressions for the
surface velocity fields describing the 
motions of the cilia. 
In this regard, it is 
important to
provide as precise as possible
parametric analysis of the controllability properties of the system.

The paper is organized as follows. Section~\ref{nots} introduces the control theoretical language adopted throughout the paper. 
Section~\ref{defs} presents the control system obtained in \cite{nancy}
and recalls the relations between the physical quantities characterizing the organism and the parameters appearing in the model. Section~\ref{abstr} is technical and contains an abstract controllability result used in the later sections.
Section~\ref{gen} is devoted to the analysis of the genericity of the
controllability of the system introduced in Section~\ref{defs}. 
Genericity is initially formulated in terms of the parameters 
appearing in the finite-dimensional system and then 
expressed in terms of the physical objects they depend on (Section~\ref{phis}).
The main result is that 
the system is generically controllable if
the number of 
controlled vector fields is at least 
three and that one controlled vector field suffices 
for its generic accessibility (Theorem~\ref{gnrc}). 
Section~\ref{tondo} analyzes the case in which the organism is spherical and provides a complete characterization of the sets of parameters that make the system controllable
(Theorem~\ref{main}). Finally,
Section~\ref{sme} establishes a characterization of controllable systems when the organism can be assumed to have  the same density as the fluid (Theorem~\ref{same}). The result follows from a further reduction of the system 
that transforms it into a six-dimensional one. 
The 
reduction procedure, suggested in \cite{nancy}, is carried out in full 
details.

\section{Notations and definitions}\label{nots}

Let us first introduce some notations: in the following $\Id$ denotes the $3\times3$ identity matrix, 
$\0_{m\times n}$ is the zero $m\times n$ matrix, $\0_m=\0_{m\times 1}$ is the zero vector in $\R^m$. We 
denote  by $\M_{m\times n}$ the set of $m\times n$ real matrices. 
The interior of a set $W$ is denoted by $\Int(W)$ and its closure by 
$\Cl(W)$.
For simplicity of notation, when no confusion is possible, row and column vectors are identified. 
The canonical orthonormal  basis of $\R^3$ is denoted by $\{e_1,e_2,e_3\}$.
For every $\om\in\R^3$ we denote by  $S(\omega)$ 
the skew symmetric matrix such that $S(\omega)z = \omega\times z$ for every $z\in\R^3$, that is, 
\[
S(\omega) = \begin{pmatrix}
0 & -\omega_3 & \omega_2\\
\omega_3 & 0 & -\omega_1\\
-\omega_2 & \omega_1 & 0
\end{pmatrix}.
\]

In what follows we deal with control systems of the type
\be\label{CS}
\dot q=f(q,u),\ \ \ \ q\in M,\ \ u\in \R^m,
\ee
where $M$ is a smooth ($\CC^\infty$) manifold and $f:M\times \R^m\to T M$ is smooth with respect to both variables. Moreover, 
for every $T\geq0$ 
the {\it endpoint mapping} $E_T$ which associates to an initial condition
$q_0\in M$ and a control function $u\in L^\infty([0,T],\R^m)$ the 
final point of the corresponding trajectory of \r{CS} 
is well defined and continuous on $M\times  L^\infty([0,T],\R^m)$, where
the $L^1$ norm is chosen on  $L^\infty([0,T],\R^m)$.
(For these and finer properties of the endpoint mapping see, for instance, 
\cite{book2}.)

Let $\F=\{f(\cdot,u)\mid u\in\R^m\}$ be the family of vector fields characterizing \r{CS}. 
For every $X=f(\cdot,u)\in\mathscr F$ we denote by $e^{t X}$ the one-parameter group generated by $X$, that is, $e^{t X}(q)=E_t(q,u 1_{[0,t]})$. 
The {\it  attainable set from $q\in M$} is the set
$$\A(q,\F)=\cup_{T\geq0}E_T(q,L^\infty([0,T],\R^m)).$$ 
When no confusion is possible we write $\A(q)=\A(q,\F)$.


We say that \r{CS} (or, equivalently, $\F$) 
is {\it 
controllable} 
 if for every $q\in M$
the set $\A(q)$ is equal to  $M$.  
We say that \r{CS} is {\it approximatively controllable}
if $\A(q)$ is  dense in $M$.

Give two vector fields $X$ and $Y$ on $M$, the {\it Lie bracket} between $X$ and $Y$ is defined, in a local system of coordinates,
by the relation  
\[
{[X,Y]}=(DY)\,X-(DX)\,Y
\]
{where $DX$ denotes the derivative of $X$.}
Denote by $\L$ the Lie algebra generated by $\F$.
System \r{CS} is {\it Lie-bracket generating} if for every $q\in M$ the set $\L(q)=\{V(q)\mid V\in\L\}$ is equal to $T_q M$.

\section{The dynamics of swimming microscopic organisms}\label{defs} \label{dynamics}

In  \cite{nancy} the authors describe 
the swimming of a microscopic organism immersed in an infinite volume of fluid 
by coupling the dynamics of a rigid body (representing the organism) with
the Navier--Stokes equations 
describing the behavior of the fluid outside the body. 
Exploiting the hypothesis of very low Reynolds number, they derive a simplified version of the coupled system. 
The 
system 
obtained in this way 
is finite-dimensional and nonlinear; more precisely, it has the following expression:
\begin{eqnarray}
\dot z&=&A z+E(z)+B u\,,\label{vz}\\
\dot \zeta&=& R\xi\label{zita}\,,\\
\dot R&=&R S(\om)\,,\label{R}
\end{eqnarray}
where $z=(\xi,\om)\in\R^3\times\R^3$, $\zeta\in\R^3$,  $R\in SO(3)$, and 
\[E(z)=\lp \ba{c}\xi\times \omega\\ J^{-1}((J\om)\times \om)\ea\rp.\]
Here $J\in \M_{3\times 3}$ denotes the inertia matrix of the rigid body representing the organism.
Recall that $J$ is symmetric and positive definite.  
The  matrix $A\in \M_{6\times 6}$ is a function of 
the rigid body and the viscosity of the fluid and its expression is given at the end of the section.

The control function
$u$ takes values in $\R^m$, $m\geq1$, and $B$ is a 6-by-$m$ matrix, which we write as
\[B=\lp\ba{c}B_1\\B_2\ea\rp\]
where both $B_1$ and $B_2$ belong to $\M_{3\times m}$.

An important feature of the system above is that 
\r{vz} is a well-defined control system on $\R^6$, since it does not depend on 
$\zeta$ nor $R$. 

Denote by $X_0$ the vector field on $\R^6\times\R^3\times\SO(3)$ which is the drift of the control system \r{vz}--\r{R}, that is, 
\[
X_0(z,\zeta,R)=\lp\begin{array}{c}
A z +E(z)\\
R\xi\\
RS(\omega)
\end{array}\rp.
\]
Let moreover $X_1,\dots,X_m$ be the controlled vector fields, 
so
that
system
\r{vz}--\r{R} can be written as
\[\dot q=X_0(q)+\sum_{i=1}^m u_i X_i(q),\ \ \ \ \ \ \ q\in \R^9\times SO(3).\]
Define $\L$ as the Lie algebra generated by $X_0,X_1,\dots,X_m$.

We say that a vector field $X$ on $\R^9\times SO(3)$ 
is a {\it constant vector field} 
if there exist  $v_1,v_2\in\R^3$ such that 
$X:(\xi,\om,\zeta,R)\mapsto (v_1,v_2,\0_6)$ and we write
$X=(v_1,v_2,\0_6)$. 
 Clearly, $X_1,\dots,X_m$ are constant vector fields. 
Notice that if $V= (v_1,v_2,\0_6)$ and $W= (w_1,w_2,\0_6)$ are constant vector fields, then
\be\label{CVF}
[[X_0,V],W]= \lp\ba{c}v_2\times w_1+w_2\times v_1\\ J^{-1}\lp Jw_2\times v_2+J v_2\times w_2\rp\\ \0_6\ea\rp 
\ee
is a constant vector field as well. 

\begin{rmk}\label{total}
For every constant vector field $V={(v_1,v_2,0_6)}$, the bracket $[X_0,V]$ is of the form 
\[ {[X_0,V]}(z,\zeta,R)={(G(z,V),-R v_1,-R S(v_2))}\] 
for some smooth map $G$ with values in $\R^6$. Therefore,
 if the constant vector fields contained in $\L$ span $\R^6\times \{0_6\}$, then
system  \r{vz}--\r{R}  is Lie-bracket generating. 
\end{rmk}

We conclude the section by recalling the analytic characterization of the coefficients appearing in $A$, $B$ and $J$. Let $\Om\subset\R^3$ be an open bounded 
connected
set whose boundary is of class $\CC^2$, representing the (rigid) shape of the organism. Then the coefficients $J_{ij}$ of $J$ are given by
\be\label{e_J}
J_{ij}=\int_{\Om} \delta(x) (e_i\times (x-\bar\xi))\cdot (e_j\times (x-\bar\xi)) d x,  
\ee
where $\bar \xi$ denotes the mass center of $\Om$ and $\delta:\Om\to(0,\infty)$ is the (possibly non-constant) density function. 

Denote by $F$ the set $\R^3\smallsetminus \Cl(\Om)$ and, for $l\in\N$, $1\leq q\leq\infty$, let
$$D^{l,q}(F)=\{\fhi\in L^1_{\mathrm{loc}}\mid \partial^\alpha \fhi\in L^q(F)\mbox{ for all }\alpha\in \N^3,\;|\al|=l\}.$$
 According to \cite[Lemma 3.1]{nancy}, for every $i\in\{1,2,3\}$ there exists a unique solution $(h^{(i)},p^{(i)})$ of
 \[
 \left\{\ba{rcll}
 -\Delta h^{(i)}+\nabla p^{(i)}&=&0&\mbox{ on }F\\
 \mathrm{div} h^{(i)}&=&0& \mbox{ on }F\\
 h^{(i)}&=&e_i& \mbox{ on }\partial \Om\\
\lim_{|y|\to\infty}h^{(i)}(y)&=&0&
 \ea\right.
 \]
 and 
 a unique solution $(H^{(i)},P^{(i)})$ of
 \[
 \left\{\ba{rcll}
 -\Delta H^{(i)}+\nabla P^{(i)}&=&0&\mbox{ on }F\\
 \mathrm{div} H^{(i)}&=&0& \mbox{ on }F\\
 H^{(i)}(y)&=&e_i\times y& \mbox{ on }\partial \Om\\
\lim_{|y|\to\infty}H^{(i)}(y)&=&0&
 \ea\right.
 \]
 such that $h^{(i)},H^{(i)}\in L^s(F)\cap D^{1,r}(F)\cap D^{2,\th}(F)\cap \CC^\infty(F)$ and $p^{(i)},P^{(i)}\in L^r(F)\cap D^{1,\th}(F)\cap \CC^\infty(F)$ for $s\in(3,\infty]$, $r\in (3/2,\infty]$ and $\th\in(1,\infty)$.
 
The {\it Cauchy stress} is the tensor $\sigma$ defined by the relation
\be\label{sigma}
\sigma(v,p)=-p\Id+\mu \lp \frac{\partial v_k}{\partial y_l}+\frac{\partial v_l}{\partial y_k}\rp_{1\leq l,k\leq 3}
\ee
where $\mu$ is the Reynolds number of the fluid.
For $i=1,2,3$, define
$$g^{(i)}=\sigma(h^{(i)},p^{(i)})n|_{\partial \Om},\ \ \ G^{(i)}=\sigma(H^{(i)},P^{(i)})n|_{\partial \Om},$$
where $n$ is the unit inner normal to $\partial \Om$. 

Denote by $\mau^1,\mau^2,\mad^1,\mad^2$ the matrices in $\M_{3\times 3}$ defined by
$$\mau^1_{ij}=-\int_{\partial \Om} g_j^{(i)}d s,\ \ \ \mau^2_{ij}=-\int_{\partial \Om}\lp x\times  g^{(i)}\rp_j ds,$$
$$\mad^1_{ij}=-\int_{\partial \Om} G_j^{(i)} ds,\ \ \ \mad^2_{ij}=-\int_{\partial \Om}\lp x\times  G^{(i)}\rp_j ds,$$
where $ds$ is the surface element of $\partial \Om$.  
It turns out (see \cite{low}) that $\mad^1=\T{(\mau^2)}$ and that
\be\label{e_A}
A=\lp\ba{cc}
{\barm}^{-1}\mau^1&{\barm}^{-1}\mau^2\\
{J}^{-1}{\mad^1}&{J}^{-1}\mad^2
\ea\rp,
\ee
where $\barm$ denotes the mass of the organism. 
Moreover, with respect to the inner product in $\R^6$ defined by
\be\label{J_scalar}
\la a,b\ra_J=\barm \sum_{l=1}^3 a_l b_l+\sum_{l,k=1}^3 J_{lk}a_{3+l}b_{3+k},
\ee
$A$ is self-adjoint and negative-definite. 

The matrix $B$ is defined by 
\be\label{e_B}
B=\lp\ba{c}
\barm^{-1}\lau\\J^{-1}\lad
\ea\rp
\ee
where the entries of $\lau$ and $\lad$ are
\[
\lau_{ij}=-\int_{\partial \Om} g^{(i)}\cdot \psi_j ds,\ \ \ \lad_{ij}=-\int_{\partial \Om} G^{(i)}\cdot \psi_j ds,
\]
and $\psi_1,\dots,\psi_m$ are fixed functions in $\CC^2(\partial \Om,\R^3)$, each one associated to a component of the control $u$, that describe the admissible 
propulsive actions of the organism.

\section{ A controllability result in a more general setting
}\label{abstr}

The following proposition
provides sufficient (and necessary) conditions for the 
controllability of systems 
that generalize those introduced in the previous section, in the sense that a more general
structure for the dynamics in the coordinates $z$ is allowed. 

\begin{prop}\label{yup}\label{GENERALE}
Consider the following control system 
\be\label{z}
\dot z=f(z,u),\ \ \ \ z\in\R^6,~~u\in\R^m.
\ee
Then \r{zita},\r{R},\r{z} is controllable if and only if 
it is Lie-bracket generating at $(z,\zeta,R)=(0_9,\Id)$ and 
system~\r{z} is controllable.
\end{prop}

\begin{proof}

One direction of the equivalence being obvious, let us 
assume that system~\r{z} is controllable and that 
\r{zita},\r{R},\r{z} is Lie-bracket generating at $(z,\zeta,R)=(0_9,\Id)$.

Notice that 
\be\label{(P)} 
(z',\zeta',R')\in \caA{(z,\zeta,R)}\ \ \ \ \Longrightarrow\ \ \ \  (z',\bar \zeta+\bar R\zeta',\bar R R')\in \caA{(z,\bar \zeta+\bar R\zeta,\bar R R)}
\ee
for every $(\bar \zeta,\bar R)\in\R^3\times SO(3)$, where the letter $\A$ is used here to denote   attainable sets for system \r{zita},\r{R},\r{z}.
Indeed, if $t\mt u(t)$ is an admissible control steering $(z,\zeta,R)$ to $(z',\zeta',R')$ 
and $t\mt (z(t),\zeta(t),R(t))$ 
is the corresponding trajectory, then $t\mt (z(t),\bar \zeta+\bar R\zeta(t),\bar R R(t))$ is an admissible trajectory corresponding to the same control $u$.

Fix $(z_0,\zeta_0,R_0),(z_1,\zeta_1,R_1)\in \R^9\times SO(3)$. Since \r{z} is 
controllable, then there exist $\eta_0,\eta_1\in\R^3$ and 
$Q_0,Q_1\in SO(3)$ such that 
$(\0_6,\eta_0,Q_0)\in \caA{(z_0,\zeta_0,R_0)}$ and 
$(z_1,\zeta_1,{R_1})\in \caA{(\0_6,\eta_1,Q_1)}$. We are left to prove that     $\caA{(\0_6,\eta_0,Q_0)}$ contains $\{\0_6\}\times \R^3\times SO(3)$.

Since \r{zita},\r{R},\r{z} is Lie-bracket generating at
$(z,\zeta,R)=(0_9,\Id)$, then the set
\[ G=\Int(\caA{(0_9,{\Id})})\]
is nonempty, as it follows from Krener's theorem (see, for instance, \cite[Theorem 1, p. 66]{Jur}).
Fix $(z_*,\zeta_*,R_*)\in G$. 
Since (\ref{z}) is 
controllable, there exists a control law $t\mapsto u(t)$, defined on an interval $[0,T]$ and such that the solution of \r{z} with initial condition $z(0)=z_*$ satisfies $z(T)=0_6$. 
The flow corresponding to the control law $u$, evaluated at time $T$, is a diffeomorphism sending $(z_*,\zeta_*,R_*)$ to a point of the form $(0_6,\zeta_{**},R_{**})$. In particular, since the image of $G$ by such diffeomorphism is contained in $G$,
there exists an open nonempty subset $V$ of $\R^3\times SO(3)$ such that $\{0_6\}\times V$ is contained in $G$.

Notice the following 
consequence of \r{(P)}: if $(0_6,\bar \zeta,\bar R)\in \caA{(\0_6,\eta_0,Q_0)}$ and $(0_6,\zeta',R')\in G$, then
\be\label{sum}
(0_6,\bar \zeta+\bar R \zeta',\bar R R')\in \Int(\caA{(\0_6,\eta_0,Q_0)})
\,.
\ee
Indeed, taking $(z,\zeta,R)=(0_9,\Id)$ in \r{(P)}, we have that 
for every $(z'',\zeta'',R'')$ in a neighborhood of $(0_6,\zeta',R')$, the point $(z'',\bar \zeta+\bar R \zeta'',\bar R R'')$ belongs to 
$\mathcal{A}(0_6,\bar\zeta, \bar R)$ and, since $(0_6,\bar\zeta,\bar R)\in\mathcal{A}(0_6,\eta_0,Q_0)$, we   are done.

Notice that for every nonempty open subset $O$ of $SO(3)$ 
\be\label{w_v}
W=\{P_1 P_2 \cdots P_k\,|\;k\in\N, P_1,\dots,P_k\in O\}
\ee
is equal to $SO(3)$.  Indeed, since $O$ contains at least one element $P_0$ of finite order, ie, an axial rotation of angle commensurable with $\pi$, then $\Id$ belongs to the interior of $W$. The completeness of $SO(3)$ 
guarantees that $W=SO(3)$.

Take as $O$ the projection of $V$ on $SO(3)$.
Then for every $R\in SO(3)$
there exists $\zeta\in\R^3$ such that
$(0_6,\zeta,R)$ belongs to $G$; \add{this can be seen by noticing 
that there exist $k$ elements $(0_6,\zeta_1,P_1),\dots,(0_6,\zeta_k,P_k)$ in 
$V$
 such that $P_i\in O$ and $P_1 P_2\cdots P_k=R$ and by 
applying repeatedly \r{sum} in the special case $\eta_0=0_3$ and $Q_0=\Id$.}

As a consequence, without loss of generality, 
$$Q_0=\Id,\ \ \ \mbox{  }\ \ \ Q_1=\Id.$$ 
\add{Indeed, since there 
exists $\zeta_0\in\R^3$ such that $(0_6,\zeta_0,Q_0^{-1})\in G$, 
equation  \r{sum} with $(\bar\zeta,\bar R)=(\eta_0,Q_0)$ shows  that    
$(0_6,\eta_0+Q_0\zeta_0,\Id)\in\caa{(0_6,\eta_0,Q_0)}$; 
similarly, the  existence of $\zeta_1\in\R^3$ such that $(0_6,\zeta_1,Q_1)\in G$ implies that $(0_6,\eta_1,Q_1)\in\caa{(0_6,\eta_1-\zeta_1,\Id)}$.
}

Moreover, for every element $p$ of $\RP^2$, the Grassmannian of  
one-dimensional subspaces of $\R^3$, there exists $\zeta_p\in\R^3$ such that $(0_6,\zeta_p,R_p)$ lies in $G$, where $R_p$ denotes the rotation of angle $\pi$ around the axis $p$.
Applying \r{sum} to $(\bar \zeta,\bar R)=(\zeta',R')=(\zeta_p,R_p)$ in the special case $\eta_0=0_3$, 
we obtain that $(0_6,\zeta_p+R_p \zeta_p,\Id)$ lies in $G$. Notice that $\zeta_p+R_p \zeta_p{=2\ps{\zeta_p}{p}p}$ belongs to the axis $p$. 
Therefore, $G$ intersects $\{0_6\}\times p\times \{\Id\}$
for every 
 $p\in\RP^2$. 

Of special interest will be for us the expression of \r{sum} when $\bar R=R'=\Id$, namely
\be\label{linear_sum}
(0_6,\bar \zeta,\Id)\in \caa{(0_6,\eta_0,\Id)},\  (0_6,\zeta',\Id)\in G\ \ \ \ \Longrightarrow\ \ \ \ (0_6,\bar \zeta+\zeta', \Id)\in \Int(\caa{(0_6,\eta_0,\Id)})\,.
\ee

Let
\[\caS^2_\pm=\{v\in \caS^2\,|\;\{0_6\}\times\R_{\geq0} (\pm v)\times \{\Id\}\cap
G\ne\emptyset\}\]
{where $\caS^2$ denotes the unit sphere in $\R^3$.}
As we remarked above, $\caS^2_+$ and $\caS^2_-$ cover $\caS^2$ and, by construction, 
they are open.
Moreover, both are 
nonempty, since $v\in\caS^2_\pm$ implies that $-v\in\caS^2_\mp$.
 Since $\caS^2$ is connected we have
 $\caS^2_+\cap\caS^2_-\ne\emptyset$, ie, there exist $v\in\R^3$ and two \add{non-negative} constants $\lb_1,\lb_2$ such that both $(\0_6,\lb_1 v,\Id)$ and $(\0_6,-\lb_2 v,\Id)$ lie in 
$G$; \add{moreover, due to the openness of $G$, these constants can be assumed to be positive and commensurable and so there exist 
two positive 
integers $n_1$ and $n_2$ such that $n_1\lb_1-n_2\lb_2=0$.}
Applying repeatedly \r{linear_sum} in the special case $\eta_0=0_3$, we obtain that 
$G$ contains $(0_9,\Id)$
and, therefore, $\{0_6\}\times\R^3\times \{\Id\}$.

Applying \r{linear_sum} again with $\bar\zeta=\eta_0$  
we deduce that $(0_6,\eta_1,\Id)$ 
is attainable from $(0_6,\eta_0,\Id)$.
\end{proof}


\section{Generic properties}\label{gen}

The expression {\it generic 
} is commonly used 
to denote 
a property of a system that is, in a suitable sense, stable under small perturbations and that, even when it fails to apply, can be forced to hold by applying to the system an arbitrary small perturbation.

In order to define precisely what a generic property means
in the present context, define
$$\Xi^m_0=\{(A,B,J)\in \M_{6\times 6}\times \M_{6\times m}\times \M_{3\times 3}\mid J=\T{J}>0,\ A\in\sym\},$$
where
$$\sym=\{A\in \M_{6\times 6}\mid A\mbox{ is symmetric and negative definite with respect to }\la\cdot ,\cdot \ra_J\}$$
and $\la\cdot ,\cdot \ra_J$ is defined as in \r{J_scalar}.

We say that a property is {\it generic for system \r{vz}--\r{R} when $m=m_0$} if there exists an open and dense set $O$ in $\Xi^{m_0}_0$ such that  
the property holds for every system defined by a triple in $O$.

Such a definition of genericity, given in terms of the 
entries 
of the matrices $A$, $J$ and $B$, 
is adapted to the finite-dimensional formulation \r{vz}--\r{R} 
but 
can appear unsatisfactory from the point of view of the physical model. 
Section~\ref{phis} discusses how to define  
the genericity directly in the space of  configurations of the organism
and provides a 
physically justified counterpart 
of the following result.

\begin{thm}\label{gnrc}
(a) Generically when $m\geq 1$ the rank of $\L$ is maximal at every point.
(b) Generically when $m\geq 3$ system \r{vz}--\r{R} is controllable.
\end{thm}

The generic conditions ensuring the Lie-bracket generating condition 
 and the controllability of \r{vz}--\r{R}
are given explicitly in the Propositions~\ref{p-a} and \ref{p-c}. 
Notice that
the techniques applied in this section are independent of the structure imposed on $A$ 
by the physical motivations of the model, namely, the fact that 
$A\in\sym$. 
Indeed, the conditions that{are} obtained in the following sections define an open and dense set in 
$$\Xi^m_1=\{(A,B,J)\in \M_{6\times 6}\times \M_{6\times m}\times \M_{3\times 3}\mid J=\T{J}>0\}$$
for $m\geq 1$ (for the Lie-bracket generating condition) or $m\geq 3$ (for the controllability) whose intersection with $\Xi^m_0$ is open and dense in $\Xi^m_0$.
In other words, Lie-bracket generating condition and controllability are generic 
properties (for $m=1$ and $m=3$, respectively) also in the class of systems of the form \r{vz}--\r{R} for which the triple $(A,B,J)$ is taken in $\Xi^m_1$.

\subsection{Generic Lie-bracket generating condition for $m=1$}\label{g-a}

When $m=1$ the two matrices $B_1$ and $B_2$ are column vectors that we denote by $b_1$ and $b_2$ respectively.
Then $X_1=({b_1},{b_2},0_6)$ and $\L$ is the Lie algebra generated by $X_0$ and 
$X_1$.

\begin{prop}\label{p-a}
Let $i=0$ or $i=1$. There exists an open and dense
 set of triples $(A,B,J)$ in 
 $\Xi^1_i$
 such that system~\r{vz}--\r{R} is Lie-bracket generating. 
 More precisely,   if $J$, $b_1$, and $b_2$ satisfy the three following conditions
\br
\det(b_1, Jb_2, J^2b_2)&\neq& 0,\label{hypo1}\\
\det(b_2,Jb_2,J^2b_2)&\neq&0,\label{hypo2}\\
\det\lp J- \dfrac{\|Jb_2\|^2}{\ps{Jb_2}{b_2}}\Id\rp&\ne&0, \label{specJ}
\er
then $\L$ is maximal at every point.
\end{prop}
\proof
Let us prove, first of all, that if $C\in\M(3\times 3)$ is symmetric and invertible, and if $b\in\R^3$ is such that $C b$ and $C^2 b$ are linearly independent, then
\begin{equation}\label{algebraic_abstract}
 C^{-1}(Cb\times b) = \det(C^{-1})(C^2 b\times Cb).
\end{equation}
Notice that  
$\ps{C^{-1}(Cb \times b)}{Cb}=0=\ps{C^{-1}(Cb\times b)}{C^2b}$, and therefore
\begin{equation}
\label{eg-vect}
C^{-1}(Cb\times b)=k(C^2b\times Cb),
\end{equation}
for some $k\in\R$. 
For every  $x\in\R^3$  we have $\det(C^{-1}x,Cb,b) = (\det C^{-1})\det(x,C^2 b,Cb)$ and  also
\begin{align*}
\det(C^{-1}x,Cb,b) & = \ps{C^{-1}x}{Cb\times b} = \ps{x}{C^{-1}(Cb\times b)}
= k\ps{x}{C^2 b\times C b} = k \det(x, C^2b,Cb).
\end{align*}
Therefore 
\[
(\det C^{-1})\det(x,C^2b,Cb) = k \det(x,C^2b,Cb)
\]
for every $x\in\R^3$, proving that \r{algebraic_abstract} holds true. 
In particular, taking $J=C$ and $b=b_2$, we have
\begin{equation}\label{algebraic}
 J^{-1}(Jb_2\times b_2) = \det(J^{-1})(J^2 b_2\times Jb_2).
\end{equation}

Denote by 
$V_1$ the Lie bracket $[[X_0,X_1],X_1]$. According to 
\r{CVF} we have  
\[
V_1=2\,((b_2\times b_1),J^{-1}(Jb_2\times b_2),\0_6).
\]
Similarly, the definition $V_2=\dfrac14[[X_0,V_1],V_1]$ 
leads to the expression
\[
V_2 = (-\ps{J^{-1}(Jb_2\times b_2)}{b_2}\, b_1 +\ps{J^{-1}(Jb_2\times b_2)}{b_1}\,{b_2},-\ps{J^{-1}(Jb_2\times b_2)}{b_2}\,{b_2},\0_6).
\]
Since $\ps{J^{-1}(Jb_2\times b_2)}{b_2}\,X_1 + V_2= 
{(\ps{J^{-1}(Jb_2\times b_2)}{b_1}\,{b_2},\0_9)}$  we 
have that
\[
Z_1= {({b_2},\0_9)}
\]
belongs to $\L$, 
as it follows from 
\r{hypo1} and from  
\r{algebraic}.

Let $Z_2=\dfrac12[[X_0,Z_1],V_1]$. 
Applying again \r{CVF}, we have that
\[
Z_2=(-b_2\times J^{-1}(Jb_2\times b_2),\0_9).
\]
Equation \r{algebraic}
implies that
\[
Z_2 = ((\det J^{-1})\,b_2\times(Jb_2\times J^2b_2),0_9).
\]
Letting 
\[
b_3=b_2\times(Jb_2\times J^2b_2)
\]
we can represent  
$Z_3=\dfrac12[[X_0,Z_2],X_1]$ as
\[
Z_3=(b_2\times b_3,\0_9).
\]
We have shown that the Lie algebra $\L$  
contains the vector fields
\[
Z_1=(b_2,\0_9),\ \ Z_2= (b_3,\0_9),\ \ Z_3 = (b_2\times b_3,\0_9).
\]
Notice that 
$$
b_3=\la b_2,J^2 b_2\ra J b_2-\la b_2,J b_2\ra J^2 b_2
$$
is not in $\spann(b_2)$, since $\la b_2,J b_2\ra>0$ and because of
\r{hypo2}.
As a consequence, the 
vectors $b_2$, $b_3$ and $b_2\times b_3$ are linearly independent.

Therefore, $\L$ contains every constant vector field 
of the type $(v,\0_9)$, with $v$ in $\R^3$. It follows 
that 
$$W_1=(\0_3,b_2,\0_6)$$ 
is in $\L$.
As a consequence of \r{CVF},  
$
[[X_0,W_1],W_1]=2(\0_3,J^{-1}(Jb_2\times b_2),\0_6)
$
lies in $\L$ as well. Due to \r{algebraic}, we deduce that
\[
W_2 = (\0_3,J^2b_2\times Jb_2,\0_6)
\]
belongs to $\L$. 
Let 
\[
b_4=J^2b_2\times Jb_2
\]
and let $W_3=[[X_0,W_1],W_2]$; we have $W_3=(0_3,b_5,0_6)$ with $b_5=J^{-1}(Jb_4\times b_2 +Jb_2\times b_4)$. 
We are left to prove 
 that the vectors $b_2$, $b_4$ and $b_5$ are linearly independent. Let $D=\det(b_2,b_4,b_5)$ and notice that
\begin{align*}
D & = \det(b_2,b_4,J^{-1}(Jb_4\times b_2 +Jb_2\times b_4))\\
& = (\det J^{-1})\det(Jb_2,Jb_4, Jb_4\times b_2 +Jb_2\times b_4)\\
& = (\det J^{-1})\ps{Jb_2}{Jb_4\times(Jb_4\times b_2) + Jb_4\times(Jb_2\times b_4)}\\
& = (\det J^{-1})\ps{Jb_2}{\ps{Jb_4}{b_4}J b_2 - \ps{Jb_4}{Jb_2}b_4 
+ \ps{Jb_4}{b_2}J b_4 - \|Jb_4\|^2b_2}.
\end{align*}
Since $\ps{Jb_4}{Jb_2}=\ps{b_4}{J^2b_2}=0$ and $\ps{Jb_4}{b_2}=\ps{b_4}{Jb_2}=0$ we have
\begin{align*}
D& = (\det J^{-1})\ps{Jb_2}{\ps{Jb_4}{b_4}J b_2 - \|Jb_4\|^2b_2}\\
& = (\det J^{-1})(\|Jb_2\|^2\ps{Jb_4}{b_4} - \|Jb_4\|^2\ps{Jb_2}{b_2}).
\end{align*}
Now, taking $C=J^{-1}$ and $b=J^2b_2$ in \r{algebraic_abstract}, we have $Jb_4 = (\det J)(Jb_2\times b_2)$, from which we deduce 
\begin{align*}
\ps{Jb_4}{b_4} & = (\det J)(\ps{Jb_2}{b_2}\ps{Jb_2}{J^2b_2} - \|Jb_2\|^4).
\end{align*}
We also have
\begin{align*}
\|Jb_4\|^2 &  = (\det J)^2\|Jb_2\times b_2\|^2\\
& = (\det J)^2(\|Jb_2\|^2\|b_2\|^2 - \ps{Jb_2}{b_2}^2).
\end{align*}
Defining $\lambda = \|Jb_2\|^2/\ps{Jb_2}{b_2}$ 
we get
\begin{align*}
D & = \|Jb_2\|^2(\ps{Jb_2}{b_2}\ps{J^2b_2}{Jb_2} - \|Jb_2\|^4) - 
(\det J)\ps{Jb_2}{b_2}(\|Jb_2\|^2\|b_2\|^2 - \ps{Jb_2}{b_2}^2)\\
& = \ps{Jb_2}{b_2}^3\left( \frac{\ps{J^3b_2}{b_2}}{\ps{Jb_2}{b_2}}\lambda - \lambda^3 
-(\det J)\left(\frac{\| b_2\|^2}{\ps{Jb_2}{b_2}}\lambda - 1\right)\right).
\end{align*}
Cayley--Hamilton theorem implies that 
\[
J^3 = a_2J^2+a_1J+(\det J)\Id,\ \ \ \ a_1,a_2\in\R,
\]
and thus
\begin{align*}
\frac{\ps{J^2b_2}{Jb_2}}{\ps{Jb_2}{b_2}} &= 
a_2\frac{\ps{J^2b_2}{b_2}}{\ps{Jb_2}{b_2}} + a_1 + (\det J)\frac{\| b_2\|^2}{\ps{Jb_2}{b_2}}\\
& = a_2 \lambda + a_1 + (\det J)\frac{\| b_2\|^2}{\ps{Jb_2}{b_2}}.
\end{align*}
Therefore, 
\begin{align*}
D & = \ps{Jb_2}{b_2}^3\left(\left (a_2\lambda + a_1 + (\det J)\frac{\| b_2\|^2}{\ps{Jb_2}{b_2}}\right)\lambda - \lambda^3 
-(\det J)\left(\frac{\| b_2\|^2}{\ps{Jb_2}{b_2}}\lambda - 1\right)\right)\\
& = \ps{Jb_2}{b_2}^3(-\lambda^3 + a_2\lambda^2 + a_1 \lambda + \det J)\\
& = \ps{Jb_2}{b_2}^3\det(J-\lambda\Id)
\end{align*}
which is different from zero because of  \r{specJ}.
\EOP

\subsection{Generic controllability in the case $m=3$}\label{g-c}

The aim of this section is to prove the genericity of the controllability 
of \r{vz}
in the case $m=3$.
Let us introduce the notation
\[
A=\begin{pmatrix}
A_{11} & A_{12}\\
A_{21} & A_{22}
\end{pmatrix}
\]
with each $A_{ij}$ belonging to $\M_{3\times 3}$.

\begin{prop}
\label{Th1}\label{p-c}
Let $i=0$ or $i=1$. 
There exists an open and dense set of triples $(A,B,J)$ in $\Xi^3_i$ such that system~\r{vz} is controllable. More precisely,  
if 
$B_2$   and $\tilde A_{11}+\T{\tilde A_{11}}$ are invertible with
$\tilde A_{11}=A_{11} -B_1B_2^{-1}A_{21}$ and if 
at least one eigenvector  of $J$ is not an eigenvector of 
$B_1B_2^{-1}$, 
then system~\eqref{vz} is controllable.
\end{prop}

Propositions~\ref{p-a} and \ref{p-c} 
lead to
 the following statement, which follows from Proposition~\ref{GENERALE}.
\begin{cor}
For $m\geq  3$ 
and $i=0$ or $i=1$, there exists an open and dense set of triples $(A,B,J)$ in $\Xi^3_i$ such that
system~\r{vz}--\r{R} is controllable.
\end{cor}

The proof of Proposition~\ref{Th1} will be split in several steps. 
The following lemma allows us to study the controllability of 
 \eqref{vz} by investigating the controllability of an equivalent 
 system in $\R^3$ (instead of $\R^6$).

\begin{lemma}
\label{propcontr3}
Assume that $B_2$ is invertible and define $\tilde B=B_1 B_2^{-1}$ and $\tilde A_{11}=A_{11} -\tilde B A_{21}$. Then, the controllability of \eqref{vz} is equivalent to the controllability of the system
\begin{equation}
\label{syst3}
\dot x  =\tilde A_{11}x - v\times x + (\tilde A_{11}\tilde B+A_{12} - \tilde B A_{22})v -v \times\tilde Bv - \tilde B J^{-1}(Jv\times v),
\quad x\in\R^3,
\end{equation}
with the control $v$ taking values in $\R^3$.
\end{lemma}
\begin{proof}
Suppose that system~\eqref{vz} is controllable. Let $x=\xi - \tilde B\omega$ and rewrite system~\eqref{vz} as
\begin{syst}
\dot x  &=\tilde A_{11}x - \omega\times x + (\tilde A_{11}\tilde B+A_{12} - \tilde B A_{22})\omega -\omega \times\tilde B\omega - \tilde B J^{-1}(J\omega\times \omega),\\
\dot\omega&= A_{21}x +( A_{21}\tilde B+A_{22})\omega + J^{-1}(J\omega\times\omega) + B_2\,u.
\end{syst}

Given two points $(x_0,\omega_0)$ and $(x_1,\omega_1)$ in $\R^6$, there exists a 
control $u:[0,T] \to\R^3$ that steers $(x_0,\omega_0)$ to $(x_1,\omega_1)$. Thus, denoting by $(x(\cdot),\omega(\cdot))$ the corresponding trajectory, the control $v(t)=\omega(t)$ steers system~\r{syst3} from $x_0$ to $x_1$.

Conversely, let system~\eqref{syst3} be controllable and fix two pairs $(\xi_0,\omega_0)$ and $(\xi_1,\omega_1)$ in $\R^6$. 
Notice that \eqref{syst3} can be controlled by smooth controls $v:[0,T]\to\R^3$
satisfying 
\be\label{b.c} 
v(0)=\om_0,\ \ \ \ \ v(T)=\om_1.
\ee
(The result follows from the controllability of \eqref{syst3}
and the density of $\{ v\in\CC^\infty ([0,T],\R^3)\mid v(0)=\om_0,\ v(T)=\om_1\}$ in $L^\infty(0,T)$ with respect to the $L^1$-norm. 
The proof can be deduced from the general results in \cite{grasse} or
easily adapted from \cite{bruno_lobry}
and \cite[Theorem 4, p. 110]{Jur}.)

Therefore, if $v\in\CC^\infty ([0,T],\R^3)$ satisfies \r{b.c} and steers \r{syst3} from 
$\xi_0-\tilde B \om_0$ to $\xi_1-\tilde B \om_1$, then 
$$u(t)=B_2^{-1}\lp\dot v(t)- A_{21}x(t) -( A_{21}\tilde B+A_{22})v(t) - J^{-1}(Jv(t)\times v(t))\rp$$
steers \eqref{vz} from $(\xi_0,\omega_0)$ to $(\xi_1,\omega_1)$.
\end{proof}

The following lemma establishes the controllability of 
system~\eqref{syst3} in a first case. 

\begin{lemma}
\label{contr1}
Assume that the 
matrix $\tilde A_{11}+\tilde A_{11}^\mathrm{T}$
has two nonzero eigenvalues with opposite signs. 
Then there exists an open and dense subset $O$
of $\{(B,J)\in \M_{6\times 3}\times \M_{3\times 3}\mid J=\T{J}>0\}$ 
such that system \eqref{syst3} is controllable if $(B,J)\in O$.  More precisely, 
if $B_2$ is invertible and if
at least one eigenvector  of $J$ is not an eigenvector of 
$B_1B_2^{-1}$, then system~\eqref{syst3} is controllable.
\end{lemma}
For the proof of this lemma we shall use the following result, proven in  \cite{JurSal}, about the controllability of a family of affine vector fields in $\R^n$. An {\it affine vector field} $X$ is a mapping from $\R^n$ to $\R^n$ of the form $X:x \mapsto C x +c$ where $C\in\M_{n\times n}$
and $c$ is a constant vector 
in $\R^n$; the {\it linear part} of $X$, denoted by $\ooverrightarrow{X}$, is the linear vector field $x\mapsto Cx$. If $\mathscr{F}$ is a family of affine vector fields, we denote by $\ooverrightarrow{\mathscr{F}}$ the set of linear parts of the vector fields in $\mathscr{F}$ and we say that $\mathscr{F}$ has no fixed point if there does not exist a point $x_0\in\R^n$ such that $X(x_0)=0$ for every $X\in\mathscr{F}$.
\begin{thm}[Jurdjevic and Sallet]
Let $\mathscr{F}
$ be a family of affine vector fields given on $\R^n$. Assume that ${\mathscr F}$ has no
    fixed point. If $\ooverrightarrow {\mathscr F}$ is controllable on $\R^n\smallsetminus\{0\}$
    then ${\mathscr F}$ is controllable on $\R^n$.
\end{thm}

\begin{proof}[Proof of Lemma~\ref{contr1}]
From Lemma~\ref{propcontr3}, we know that it is sufficient to prove the controllability of the system 
defined by the family of affine vector fields 
$
\mathscr{F}=\{\,C(v)x+c(v)\ |\ v\in\R^3\,\}
$
 where
\begin{align*}
C(v) & =\tilde A_{11} - S(v), & c(v) &=(\tilde A_{11}\tilde B+A_{12} - \tilde B A_{22})v -v\times\tilde Bv - \tilde B J^{-1}(J v\times v).
\end{align*}

Choose $v_0$ in such a way that $S(v_0) = (\tilde A_{11} - \tilde A_{11}^\mathrm{T} )/2$. Then 
$$\tilde A_{11} - S(v_0)=\frac{\tilde A_{11} + \tilde A_{11}^\mathrm{T}}2.$$

Define the family of linear vector fields $\mathscr{G}=\{\,\tilde A_{11}-S(v_0),S(v)\mid v\in\R^3\,\}$.
The two closed convex cones generated by the families $\ooverrightarrow{\mathscr{F}}$ and $\mathscr{G}$ are identical, since for every finite family $(\alpha_i)_{0\le i\le N}$  of positive numbers
and every choice of $v_1,\dots,v_N\in \R^3$,
\begin{align*}
\alpha_0(\tilde A_{11} - S(v_0)) + \sum_{i=1}^N\alpha_i S(v_i)  & = \alpha_0(\tilde A_{11} - S(v_0))  + 
\lim_{r\to+\infty}\frac1r\sum_{i=1}^N\alpha_i\bigl(\tilde A_{11} - S(-r\,v_i)\bigr)
\end{align*}
and 
\begin{align*}
\sum_{i=1}^N\alpha_i(\tilde A_{11} - S(v_i)) & = \left(\sum_{i=1}^N\alpha_i\right)
\bigl(\tilde A_{11}-S(v_0)\bigr)+ \sum_{i=1}^N\alpha_iS(v_0-v_i).
\end{align*}
 Therefore, 
the controllability of $\Lambda(\mathscr{F})$ on $\R^3\smallsetminus\{0\}$ is equivalent to the one of $\mathscr G$ (see \cite{JK1,JK2}). 

The trajectories of
the vector field $S(v)$ are circles contained in planes orthogonal to $v$ and whose centers are at the intersections of these planes with the line $\R v$. Let $x\in\R^3\smallsetminus\{0\}$; thanks to the vector fields $S(v)$, the attainable set from $x$ for $\mathscr{G}$ contains the sphere of center $\0_3$ passing through $x$. 
Now, as $\tilde A_{11}-S(v_0)$ has two eigenvalues with opposite sign, we can move (thanks to this vector field) along a direction towards the origin and 
along an half-line exiting the sphere and going to infinity.   
Therefore, using again the fact that the family $\{S(v)\mid v\in\R^3\}$ is transitive on 
every sphere, we proved that $\mathscr{G}$, and thus  $\ooverrightarrow{\mathscr F}$, 
is controllable on  $\R^3\smallsetminus\{0\}$.

Assume now that we can find a fixed point, denoted by $x_0$, which is common to all the vector fields in $\mathscr F$. For every vector $v\in\R^3$, we have 
\[
C(v)x_0 + c(v) = C(-v)x_0 + c(-v)=\0_3
\]
from which we deduce 
\be\label{fine}
\tilde A_{11}x_0- v\times\tilde Bv - \tilde BJ^{-1}(Jv\times v) = 0.
\ee
Taking $v=\0_3$ gives $\tilde A_{11}x_0=\0_3$ 
and therefore \r{fine} can be rewritten as
\begin{equation}
\label{valpropBJ}
\tilde B v\times v = \tilde BJ^{-1}(Jv\times v) 
\end{equation}
for every $v\in\R^3$. 
In particular, if $v$ is an 
eigenvector of 
$J$ but not of $\tilde B$, then 
the right-hand side of \r{valpropBJ} is equal to zero, while the 
right-hand side is not, leading to a contradiction. 
\end{proof}

We are left to deal with the case where the 
eigenvalues  of 
$$\tilde A_{11}^s=\frac{\tilde A_{11}+ \tilde A_{11}^\mathrm{T}}2,$$
the
symmetric part of $\tilde A_{11}$, are all positive or all negative. 
Although in this case the linear part of the family $\mathscr F$ is not controllable on $\R^3\smallsetminus\{0\}$, 
we can 
nevertheless adapt the method introduced in 
\cite{JurSal} in order to prove the controllability of \r{syst3}.

\begin{lemma}
\label{contr2}
Assume that the matrix $\tilde A_{11}^s$ is invertible and that its eigenvalues  
have all the same sign. 
Then there exists an open and dense subset $O$
of $\{(B,J)\in \M_{6\times 3}\times \M_{3\times 3}\mid J=\T{J}>0\}$ 
such that system \eqref{syst3} is controllable if $(B,J)\in O$.
  More precisely, if $B_2$ is invertible and if
at least one eigenvector  of $J$ is not an eigenvector of 
$\tilde B$, then system~\eqref{syst3} is controllable.
\end{lemma}
Notice that, together with Lemma~\ref{contr1},  Lemma~\ref{contr2} concludes the proof of 
Proposition~\ref{Th1}. In order to show Lemma~\ref{contr2}  we need   
some definitions and have to
prove some intermediate results.

Let $\mathscr{F}$ and $\LL{(\mathscr{F})}$ be
defined   as above.
In the proof of Lemma~\ref{contr1} 
we pointed out that starting form a point $x\in\R^3$ and following all possible 
linear vector fields of the type $S(v)$ one can attain 
the entire sphere of center $0_3$ and radius $\|x\|$. 
 If the eigenvalues of $\As$ are negative (resp. positive), the vector field $\As$ points towards the interior (resp. the exterior) of this sphere, from which we deduce that 
$$\A(x,\LL{(\mathscr{F})})=\left\{\ba{ll}
\{x\}\cup \{y\in\R^3\mid 0<\|y\|<\|x\|\}&\mbox{ if }\As<0,\\
\{x\}\cup \{y\in\R^3\mid \|y\|>\|x\|\}&\mbox{ if }\As>0.
\ea\right.
$$ 
 In the sequel, we shall need the notion of normal accessibility, 
 which is recalled below.

\begin{deff}
Let $\mathscr V$ be a family of complete vector fields on a manifold $M$. The point $y$ is said to be normally $\mathscr V$-accessible from $x$ if there exist $X^1,\dots, X^p$ in $\mathscr V$ and 
$t_1,\dots,t_p >0$ such that 
\[
y = e^{t_1 X^1}\circ\dots\circ e^{t_p X^p}(x)
\]
and the mapping 
\[
(\tau_1,\dots,\tau_p) \longmapsto e^{\tau_1 X^1}\circ\dots\circ e^{\tau_p X^p}(x),
\]
defined in a neighborhood of $(t_1,\dots, t_p)$, is of rank equal to $\dim M$ at $(t_1,\dots,t_p)$.
\end{deff}

In \cite{JurSal} the authors exploit the fact that if a family of vector fields is controllable, then every point $y$ is normally accessible from every point $x$. As the family $\ooverrightarrow{\mathscr F}$   is 
not  controllable on $\R^3\smallsetminus\{0\}$, 
we are lead 
to prove directly the following  normal accessibility property.
 
\begin{lemma}\label{spheres}
Assume that the eigenvalues of  $\tilde A_{11}+\T{\tilde A_{11}}$ have all the same sign. 
Then we can extract
from $\F$ a finite family ${\F^0}$ such that there exists a sphere of center $\0_3$ and radius $r$ of points normally $\ooverrightarrow{\F^0}$-accessible from every point of the unit sphere. 
\end{lemma}

\begin{proof} 
Let $v_0\in\R^3$ be such that $\As=\tilde A_{11}-S(v_0)=\As$.
Let, moreover, $w_1$ be an  eigenvector of $\As$ and take $w_2,w_3\in \R^3$ such that the family $(w_1,w_2,w_3)$ is an orthogonal basis of $\R^3$. Consider now the family 
\[
\ooverrightarrow{\F^0}=\{\,\As, \As- k_2 S(w_2),\As-k_3 S(w_3)\,\}
\]
 extracted from $\ooverrightarrow{\F}$. 
 
 Since $S(w_2)$ and $S(w_3)$ are nonzero skew-symmetric  matrices, each of their spectra contains the value zero and two
 nonzero purely imaginary eigenvalues. 
Notice that  
\[
\lim_{k_j\to+\infty}  \frac{\tilde A_{11} -k_jS(w_j)}{k_j} = -S(w_j).
\]
As the roots of a characteristic polynomial depend continuously on the coefficients of the related matrix, it follows that for $k_2$ and $k_3$ large enough $\As- k_2 S(w_2)$ and $\As-k_3 S(w_3)$  have a pair of non-real eigenvalues. By projecting the vector fields $\As-k_2S(w_2)$ and $\As - k_3S(w_3)$ on the unit sphere, it is easy to 
 see that the family  $\ooverrightarrow{\F^0}$ is transitive on the directions, that is, from any half-line starting from the origin, one can reach any other half-line. Moreover, thanks to the eigendirection $w_1$ of the vector field $\As$ we can go as far as we want (case $\As>0$) or as close as we want to the origin (case $\As<0$). 
Hence,
for every $x\in 
 \R^3\smallsetminus\{\0_3\}$,  
 $\A(x,\ooverrightarrow{\F^0})$ contains a set of the form $\R^3\smallsetminus B(0,r)$ (case $\As>0$) or $B(0,r)\smallsetminus\{\0_3\}$ (case $\As<0$). The interior of 
 $\A(x,\ooverrightarrow{\F^0})$ being nonempty, there exists at least one point, denoted by $y$, that can be normally $\ooverrightarrow{\F^0}$-accessed from $x$ (see \cite{Sus}). Notice that
since $y$ is normally $\ooverrightarrow{\F^0}$-accessible from $x$, then $y$ is also normally $\ooverrightarrow{\F^0}$-accessible from every point in a (sufficiently small) neighborhood of $x$. Moreover, 
all the points in $\A(y,\ooverrightarrow{\F^0})$ are normally $\ooverrightarrow{\F^0}$-accessible from $x$.

We conclude by using the compactness of the unit sphere: for every $x\in \caS^2$, there exists a neighborhood $V_x$ of $x$ and a set $\R^3\smallsetminus B(0,r_x)$ (case $\As>0$) or $B(0,r_x)\smallsetminus\{\0_3\}$ (case $\As<0$) whose points are normally $\ooverrightarrow{\F^0}$-accessible from every point of $V_x$. As we can include the unit sphere in a finite union of neighborhoods $V_x$, we can claim the existence of a sphere of radius $r$ whose points are  normally $\ooverrightarrow{\F^0}$-accessible from every point of the unit sphere.
\end{proof}

The following lemma 
guarantees 
the unboundedness of the sets $\A(x,\F)$ and $\A(x,-\F)$.

\begin{lemma}\label{unbondedness}
If $B_2$ is invertible and if
at least one eigenvector  of $J$ is not an eigenvector of 
$\tilde B$, then for every $x\in\R^3$ both sets $\A(x,\F)$ and $\A(x,-\F)$ are unbounded.
\end{lemma}
\begin{proof}
Take $v\in\R^3$ such that $q(v)= -v\times \tilde Bv - \tilde B J^{-1}(Jv\times v)\neq \0_3$; this is possible since at least one eigenvector  of $J$ is not an eigenvector of $\tilde B$. 
For every $\alpha\in\R$ the vector field
$x\mapsto C(\al v)x+c(\al v)$ is in  $\F$ and 
\[
\lim_{\alpha\to\infty}\frac{C(\al v)x+c(\al v)}{\al^2}=q(v).
\]
This proves that, given a time $T>0$, the solution of $\dot x=C(\al v)x+c(\al v)$ on the interval $[-T\al^{-2},T\al^{-2}]$ is as  close as we want to the solution of  $\dot x=q(v)$ on $[-T,T]$ (with the same initial condition) 
provided that $\alpha$ is large enough. 
Now, the solution of $\dot x = q(v)$, $x(0)=x_0$,  
leaves any fixed bounded set both in time $T$ and in time $-T$, provided 
that $T$ is large enough. 
\end{proof}

We consider now, as in \cite{JurSal}, the  
family $\{h_{\lambda,w}\mid \lb>0,\ w\in\R^n\}$ of affine diffeomorphisms of $\R^n$  
 defined by $h_{\lambda,w}(x) =w+ \lambda(x-w)$. 
 An easy computation shows that,
 for every $X\in\F$, 
\[
({h_{\lambda,w}}_*\,X)(x) = \ooverrightarrow{X}(x-w) + \lambda\,X(w)
\]
where ${h_{\lambda,w}}_* X$ denotes the pushforward  of the vector field $X$ by $h_{\lambda,w}$. In particular 
\begin{equation}
\label{lim}
\lim_{\lambda\to0}({h_{\lambda,w}}_*\,X)(x) = \ooverrightarrow{X}(x-w).
\end{equation}
Denote by $\F_{\lambda,w}$ the image of the family $\F$ by ${h_{\lambda,w}}_*$, that is, the transformation of $\F$  under the change of coordinates $h_{\lambda,w}$.

\begin{proof}[Proof of Lemma~\ref{contr2}]
Let $\ooverrightarrow{\F^0}$ and $r$ be as in the statement of 
Lemma \ref{spheres}. 
Consider the finite family $\ooverrightarrow{\F^0_{w}}= \{\,\ooverrightarrow{X}(x-w)\mid {X}\in{\F^0}\,\}$ and notice that
each point of the sphere of center $w$ and radius $r$ is normally $\ooverrightarrow{\F^0_w}$-accessible from every point of the sphere of center $w$ and radius 1.
Let 
%
%
$\F_{\lambda,w}=\{{h_{\lambda,w}}_* X \mid X\in \F\}$. 
Thanks  to formula~\eqref{lim} we can assert that, if $\lambda$ is chosen sufficiently small, the sphere of center $w$ and radius $r$ is contained in the reachable set 
for  $\F_{\lambda,w}$ 
from every point of the sphere of center $w$ and radius $1$ (see \cite[Lemma 3.2]{Sus}). 
Fix such a $\lambda>0$.

Given $w$ in $\R^3$ we claim that $\A(w,\F)$ contains a neighborhood of $w$. 
Indeed, 
since the sets $\A(x,\F)$ and $\A(x,-\F)$ are unbounded, the same is true for the sets $\A(x,\F_{\lambda,w})$ and $\A(x,-\F_{\lambda,w})$, for every $x\in\R^3$. This implies in particular, because of the arc-connectedness of attainable sets,  
 that there exists 
$y\in \A(w,\F_{\lambda,w})$ such that $\|y-w\| = 1$. 
Let $\rho<r$ and fix $z\in B(w,\rho)$. 
Since $\A(z,-\F_{\lambda,w})$ is unbounded, 
 then, again by  arc-connectedness,
there exists $z'\in \A(z,-\F_{\lambda,w})$ such that $\|z'-w\|=r$.
Since from every point of the sphere of center $w$ and radius 1, we can reach any point of the 
sphere of center $w$ and radius $r$, the point $z'$ is reachable from $y$ by the family $\F_{\lambda,w}$. 
\begin{figure}
\label{slere}
\begin{center}
\input{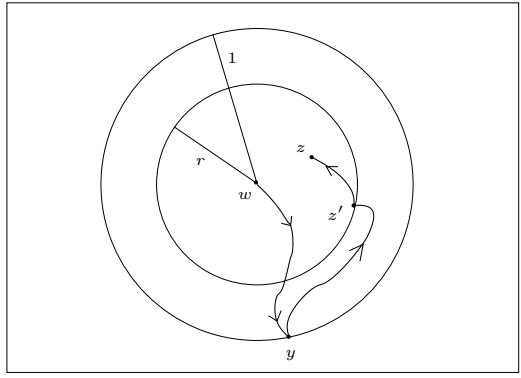}
\caption{$z$ belongs to $\A(w,\F_{\lambda,w})$ (the case $r<1$).}
\end{center}
\end{figure}
Finally,
every point $z\in B(w,\rho)$ belongs to $\A(w,\F_{\lambda,w})$ (see Figure~\arabic{figure}).
Since $\F_{\lambda,w}$ is the transformation of $\F$ by a diffeomorphism preserving $w$, we deduce that $\A(w,\F)$ contains a neighborhood of $w$. Therefore, every attainable set $\A(x,\F)$ is open.

Since the control system defined by $-\F$ has the same form as that defined by $\F$ and verifies the hypotheses of both Lemmas \ref{spheres} and \ref{unbondedness}, 
the reasoning above proves also that every attainable set $\A(x,-\F)$ is open.
Take now $y$ in the closure of $\A(x,\F)$. The set $\A(y,-\F)$ being open, there exists 
$z\in \A(y,-\F)\cap \A(x,\F)$ which proves that $y$ is reachable from $x$ by the family $\F$. 
Every attainable set $\A(x,\F)$ is 
therefore both open and closed; thus it is equal to $\R^3$.
\end{proof}

\subsection{Generic controllability: physical 
interpretation
%
}\label{phis}

The scope of this section is to provide a better physical insight of Theorem~\ref{gnrc}.
More precisely, we want to check that for $m\geq 3$  
there exists an open and dense set in the (suitably defined) space of 
microscopic organisms 
modeled here
such that the corresponding control system 
\r{vz}-\r{R} is controllable.

The space of organisms, denoted by $\Sigma_m$, will be identified with 
the set of pairs $(\Om,\Psi)$, where $\Om\subset \R^3$ is open, nonempty, connected, bounded, and
of class $\CC^2$, while $\Psi$ is the $m$-uple of functions
in $\CC^2(\partial \Om,\R^3)$ 
determining the action of the controls
(see Section~\ref{dynamics}). For simplicity 
we assume that the density $\delta$ of the organism is constant, so that the inertia matrix  $J$ is determined (up to the positive multiplicative constant $\delta$) by the shape $\Om$. (For the general case see Remark~\ref{nonC}.)
The topology on  $\Sigma_m$ can be defined assigning 
the 
basis 
of open subsets
defined by
\[
\begin{split}
\mathscr{N}_{m,\eps}(\Om,\Psi)=&\left\{\lp (\Id+\upsilon)(\Om),\Phi\rp\mid \upsilon\in W^{3,\infty}(\Om,\R^3),\Phi\in(\CC^2(\partial \Om,\R^3))^m,\right.\\
&\phantom{\{}\left.\|\upsilon\|_{3,\infty}<\eps,\|\Phi\circ(\Id+\upsilon)-\Psi\|_{\CC^2(\Om)}<\eps\right\}
\end{split}
\]
for all $\eps\in(0,1/2)$ and all $(\Om,\Psi)\in \Sigma_m$
(see \cite{simon} or, for a different approach,  \cite{chitour_coron_garavello}). 
By extension $\Sigma_0$ denotes the set of $\CC^2$, open, nonempty, connected, bounded subsets of $\R^3$ endowed with the topology whose basis is given by all
\[
\mathscr{N}_{0,\eps}(\Om)=\left\{ (\Id+\upsilon)(\Om)\mid 
\upsilon\in W^{3,\infty}(\Om,\R^3),\;
\|\upsilon\|_{3,\infty}<\eps\right\},
\]
$\eps\in(0,1/2)$, $\Om\in \Sigma_0$.

\begin{thm}
Assume that $m\geq 3$. 
There exists an open and dense set 
${\cal O}$ in $\Sigma_m$ such that \r{vz}-\r{R} is controllable
if $(\Om,\Psi)$ belongs to ${\cal O}$. Moreover, there exists an open and dense set ${\cal P}$ in $\Sigma_0$ such that
${\cal O}\cap (\{\Om\}\times (\CC^2(\partial \Om,\R^3))^m)$
is dense in  $\{\Om\}\times (\CC^2(\partial \Om,\R^3))^m$ for every $\Om$ in ${\cal P}$.
\end{thm}
\begin{proof} 
Denote by $F(\Om,\Psi)$ the triple $(A,B,J)$  associated to  
an element $(\Om,\Psi)$ of $\Sigma_0$ through \r{e_J}, \r{e_A} and \r{e_B}.
We want to show that 
each condition appearing in the 
statement of
Propositions~\ref{p-a} and \ref{p-c}
is satisfied by the elements of the image through $F$ of
an open and dense subset of $\Sigma_0$.

Notice that the map  $\Psi \mt B$ is onto when considered from $(\CC^2(\partial \Om,\R^3))^m$ to $\M_{6\times m}$ ($\Om$ fixed). Indeed,
according to \cite[Lemma~4.1]{nancy}, $g^{(1)}, g^{(2)}, g^{(3)},G^{(1)}, G^{(2)}, G^{(3)}$ are linearly independent in $L^2(\partial \Om,\R^3)$. Therefore,
 since the orthogonal 
to a smooth function 
in $\CC^2(\partial \Om,\R^3)$ with respect to the $L^2$-product has infinite codimension, 
the relation \r{e_B} defines a surjective map.
In particular, the pre-image of an open and dense set in
$\M_{6\times m}$ is open and dense in $(\CC^2(\partial \Om,\R^3))^m$.

Let us now take into account the dependence on $\Om$. 
We claim 
that the map from $\Sigma_0$ 
to  $\R^3$ that associates to a domain $\Om$ the spectrum of the corresponding inertia matrix --with eigenvalues repeated according to their multiplicity and with $\delta$ normalized to one-- is 
 locally open at every point.
 In order to check it, 
 fix $\Om\in \Sigma_0$ and a system of coordinates such that $0_3$ is the center of mass of $\Om$ and $e_1,e_2,e_3$ its principal axes of inertia. 
 Then the elements of the normalized inertia matrix $J^0=J/\delta$ of $\Om$ are
\be\label{centrato}
J^0_{ij}=\left\{
\ba{ll}
\pm \int_\Om x_i x_j dx=0&\mbox{ if }i\ne j,\\
 \int_\Om (\sum_{k\ne i}x_k^2)d x&\mbox{ if }i= j. 
\ea\right.
\ee 
 Apply the transformation $\Id+\upsilon$ to $\Om$, where $\upsilon$ is the diagonal matrix $\diag(\upsilon_1,\upsilon_2,\upsilon_3)$, and denote by $J^0(\upsilon)$ the normalized inertia matrix of $(\Id+\upsilon)(\Om)$.  Then 
$$
J^0_{ij}(\upsilon)=\left\{
\ba{ll}
\pm (\Pi_{k=1}^3(1+\upsilon_k))(1+\upsilon_i)(1+\upsilon_j)\int_\Om x_i x_j dx=0&\mbox{ if }i\ne j,\\
\pm (\Pi_{k=1}^3(1+\upsilon_k)) \int_\Om (\sum_{k\ne i}(1+\upsilon_k)^2 x_k^2)d x&\mbox{ if }i= j. 
\ea\right.
$$
 The spectrum of $J^0(\upsilon)$ is therefore given by the components of the vector $\sigma^0(\upsilon)=(\pm (\Pi_{k=1}^3(1+\upsilon_k))\int_\Om (\sum_{k\ne i}(1+\upsilon_k)^2 x_k^2)d x)_{i=1}^3$. A straightforward  computation shows that the determinant of the Jacobian matrix of $\sigma^0$ with respect to $(\upsilon_1,\upsilon_2,\upsilon_3)$, computed at  $\upsilon_1=\upsilon_2=\upsilon_3=0$, is different form zero. The map $\upsilon\mapsto \sigma^0(\upsilon)$ is therefore a submersion at 
 $\upsilon=0$.
   
 In particular, 
 the eigenvalues of $J$ are simple for $\Om$ in an open and dense subset 
 of $\Sigma_0$, independently of $\Psi$. If $b_2$ does not belong to any of the three planes generated by pairs of eigenvectors of $J$, then \r{hypo2} is automatically verified. 
 
Similarly, one notices that all 
the  assumptions
appearing in the statements of  
 Propositions~\ref{p-a} and \ref{p-c}
can be represented in the form 
$G(\Om,\Psi)\in O$ where
 $O$ is an open and dense subset of a finite-dimensional vector space $V$, the map $G:\Sigma_m \to V$ is continuous and
  for $\Om$ in an open and dense subset of $\Sigma_0$
  we have that $G(\Om,\Psi)$ belongs to ${\cal O}$
  if $B$ belongs to an open and dense subset of $\M_{6\times m}$ (possibly depending on $\Om$).  
Therefore,  $G^{-1}({\cal O})$ is open and dense in $\Sigma_m$.
\end{proof}

\begin{rmk}
As recalled above it follows from \cite{nancy} that, once $\Om$ is fixed, the linear map associating 
to the $m$-uple $\Psi$
the matrix $B$ through \r{e_B}
is onto as a map from $(L^2(\partial \Om,\R^3))^m$ to $\M_{6\times m}$. 
Thanks to this result, 
and to the remark that \r{vz}-\r{R} is controllable if $m=6$ and $B$ is invertible,
 the authors prove that   for $m\geq 6$, for an open and dense subset of $m$-uples in $\CC^2(\partial \Om,\R^3)$ ($\Om$ fixed) the corresponding system
is controllable (see \cite[Theorem~1.1]{nancy}). In the 
language adopted here the result of \cite{nancy} says that, for every fixed $\Om$ and for $m\geq 6$, controllability is a generic property with respect to $\Psi$. 
\end{rmk}

\begin{rmk}\label{nonC}
In order to introduce a reasonable notion of genericity in
the case where the density $\delta$ of the organism is not constant, 
it is necessary to include $\delta$ in 
the definition of $\Sigma_m$. 
Let $\hat \Sigma_m$ be the set of triples $(\Om,\Psi,\delta)$ where $(\Om,\Psi)\in\Sigma_m$ and $\delta\in L^\infty(\Om,(0,+\infty))$ and endow $\hat \Sigma_m$ with the topology whose basis is given by all
\[
\begin{split}
\hat{\mathscr{N}}_{m,\eps}(\Om,\Psi,\delta)=&\left\{\lp (\Id+\upsilon)(\Om),\Phi,\gamma\rp\mid \upsilon\in W^{3,\infty}(\Om,\R^3),\Phi\in(\CC^2(\partial \Om,\R^3))^m,
\gamma\in L^\infty(\Om,(0,+\infty)),
\right.\\
&\phantom{\{}\left.\|\upsilon\|_{3,\infty}<\eps,\|\Phi\circ(\Id+\upsilon)-\Psi\|_{\CC^2(\Om)}<\eps,\|\gamma\circ(\Id+\upsilon)-\delta\|_{L^\infty(\Om)}<\eps
\right\},
\end{split}
\]
$\eps\in(0,1/2)$, $(\Om,\Psi,\delta)\in \hat\Sigma_m$. 
Following the same arguments of proof as above, one can prove that 
if $m\geq 3$ there exists an open and dense set 
${\cal O}$ in $\hat\Sigma_m$ such that \r{vz}-\r{R} is controllable
if $(\Om,\Psi,\delta)$ belongs to ${\cal O}$.
The only difference consists in verifying that the map  associating to $(\Om,\delta)$ the spectrum of $J$ is locally open, and this can be done by taking the same perturbation $(\Id+\diag(\upsilon_1,\upsilon_2,\upsilon_3))(\Om)$ as above for $\Om$ and the perturbation $\delta\circ (\Id+\diag(\upsilon_1,\upsilon_2,\upsilon_3))$ for $\delta$. 
\end{rmk}

\section{Swimming spherical microscopic organisms}\label{tondo}
When the microscopic organism happens to be spherical, the equations presented in the previous sections have a very special form.
As described in \cite{nancy}, general results in hydrodynamics 
(see, 
e.g.,  \cite{low})
show that the matrix $A$ 
corresponding to a spherical organism is diagonal and, more precisely, of the form
\[A=\lp\ba{cc}-\rho_1 \Id& 0\\ 0&-\rho_2 \Id\ea\rp,\]
with 
\be\label{rhos}
\rho_2>\rho_1>0. 
\ee
Moreover, since the inertia matrix 
$J$ 
is proportional to the identity (assuming that the mass distribution is homogenous inside $\Om$), the nonlinear term $E$ appearing in \r{vz}
is given by
\[E(z)=\lp\ba{c}\om\times \xi\\ 0_3\ea\rp.\]

Notice that in the spherical case 
$\om$ satisfies a well-defined control subsystem, ie,
\be\label{eq_om}
\dot \om=-\rho_2 \om +B_2 u.
\ee

We are going to prove the following result.
\begin{thm}\label{main}
Let $\Om$ be a ball and assume that the mass distribution of $\Om$ is homogenous.
Then the control system \r{vz}--\r{R} is 
controllable if 
and only if 
the rank of $B_2$ is equal to three and $B_1$ is different from $\0_{3\times m}$. 
\end{thm}

The easier part of the proof is to show 
that the 
controllability of  \r{vz}--\r{R}
implies that 
\be\label{f_rank}
\rk B_2=3,\ \ \ \ \ B_1\not=0_{3\times m}.
\ee 
Indeed, 
the linear control system 
\r{eq_om}
is, as it is well known, controllable if and only if $\rk B_2=3$.
Moreover, 
if $B_1=0_{3\times m}$, then
the space $\{0_3\}\times\R^3$ is invariant for the dynamics of
the control system \r{vz}, which is therefore non-controllable.

The converse implication 
will be proven in several steps.

\subsection{The case where $B_1$ and $B_2$ are linearly independent}
 In this section 
we study the case where the organism is spherical and
\be\label{rank}
\rk B_2=3,\ \ \ \ \ B_1\not\in\spann(B_2).
\ee

Let us prove the following technical result.
\begin{lemma}
If assumption \r{rank} holds true, then 
there exist $c_1,c_2,c_3\in 
 \R^3$,
 an orthonormal basis  $\{d_1,d_2,d_3\}$ of $\R^3$,
and  a $m\times  3$ matrix $\Gamma$ such that
\be\label{the_last}
B_1 \Gamma v=\sum_{i=1}^3 v_i c_i,\ \ \ \ B_2 \Gamma v=\sum_{i=1}^3 v_i d_i,\ \ \ \  c_1\times d_1\ne\0_3.
\ee
\end{lemma}
\begin{proof}
 Firstly, let us consider the case $m=3$. Then $B_2$ is invertible and for every orthonormal basis $\{d_1, d_2,d_3\}$ of $\R^3$, we have
 $$ B_2 \Gamma v=\sum_{i=1}^3 v_i d_i,$$
 for $\Gamma=B_2^{-1}(d_1|d_2|d_3)$, where $(d_1|d_2|d_3)$ denotes the matrix whose columns are $d_1$, $d_2$ and $d_3$. 
 Since $B_1 B_2^{-1}$ is not proportional to $\Id$, then 
 $d_1$ can be chosen such that $d_1\times (B_1 B_2^{-1}d_1)\ne 0_3$. 
 It suffices to complete $d_1$ to any orthonormal basis $\{d_1, d_2,d_3\}$, since $B_1 \Gamma v=\sum_{i=1}^3 v_i c_i$ with 
 $c_i=B_1 B_2^{-1} d_i$.

Let now $m>3$. 
Denote by $B_i^{(j)}$ the $j$-th column of $B_i$. 

The second case we consider is when
there exist $1\leq j_1,j_2,j_3\leq m$ such that
$B_2^{(j_1)}$, $B_2^{(j_2)}$, $B_2^{(j_3)}$ 
are linearly independent and 
$(B_1^{(j_1)}|B_1^{(j_2)}|B_1^{(j_3)})\not\in\spann(B_2^{(j_1)}|B_2^{(j_2)}|B_2^{(j_3)})$.  Then, replacing $v$ by $\Gamma' v=(v_{j_1},v_{j_2},v_{j_3})$, we fulfill the hypotheses of the case 
$m=3$ and we are done.

Finally, consider the case in which for every triple $1\leq j_1,j_2,j_3\leq m$ such that
$B_2^{(j_1)}$, $B_2^{(j_2)}$, $B_2^{(j_3)}$ 
are linearly independent, we have that $(B_1^{(j_1)}|B_1^{(j_2)}|B_1^{(j_3)})$ is proportional to $(B_2^{(j_1)}|B_2^{(j_2)}|B_2^{(j_3)})$. Let 
$$\{j_1,\dots,j_k\}=\{j\mid 1\leq j\leq m,\;B_2^{(j)}\ne 0_3\}.$$
Then $(B_1^{(j_1)}|\cdots|B_1^{(j_k)})$ is proportional to 
$(B_2^{(j_1)}|\cdots|B_2^{(j_k)})$. 
Assumption~\r{rank} implies that there exists $1\leq j\leq m$ such 
that $B_1^{(j)}\not=0_3$ and $B_2^{(j)}=0_3$. 
It suffices then to replace $v$ by $\Gamma' v=(v_{j_1}+v_j,v_{j_2},\dots,v_{j_k})$ in order to satisfy the hypotheses of the previous case. 
\end{proof}

Notice now that 
for every $Q\in SO(3)$, the change of coordinates $\xi'=Q\xi,\om'=Q\om,\zeta'=\zeta,R'=R Q^{-1}$ preserves the dynamics of \r{vz}--\r{R}.
This is equivalent to say that, up to a change of coordinates, we can replace $(B_1,B_2)$ by $(Q B_1,Q B_2)$, for every choice of $Q\in SO(3)$. 
Let $d_1,d_2,d_3$ be as in \r{the_last} and 
choose $Q$ such that $Q d_i=e_i$ for $i=1,2,3$ \add{(where we recall that $(e_i)_{i=1,2,3}$ denotes the canonical basis in $\R^3$)}. 
The above transformations show that 
proving the 
controllability 
of the control system \r{vz}--\r{R} under assumption~\r{rank}
is equivalent to proving it under the hypotheses that $m=3$, $B_2=\Id$, and that the first column of $B_1$ is not proportional to $e_1$.

Denote by $b_1$, $b_2$, $b_3$ the columns of $B_1$. 
Accordingly,
the controlled vector fields 
are the three constant vector fields
$X_i=(b_i,e_i,\0_6)$, $i=1,2,3$.


\begin{lemma}\label{access}
If \r{rank} holds then system \r{vz}--\r{R} is Lie-bracket generating. 
\end{lemma}
\prf\ 
%
Applying \r{CVF} in the case $J=\Id$ we obtain that 
for every $v_1,v_2\in\R^3$ the constant vector field $V={(v_1,v_2,0_6)}$ satisfies
\be\label{croc}
  {[[X_0,X_i],V]}={(e_i\times v_1-b_i\times v_2,0_9)}.   
\ee
In particular,  taking $i=1$ and $V=X_1$ in \r{croc}, we deduce that  
$Z={(e_1\times b_1,0_9)}$ belongs to $\L$. 
Applying again \r{croc} to $V=Z$ we get that, for $i=1,2,3$,
$Z_i={( e_i\times(e_1\times b_1),0_9)}$ belongs to $\L$. Since 
$Z$, $Z_1$, $Z_2$, and $Z_3$ span $\R^3\times\{0_9\}$, it follows that $\L$
contains 
$$\spann(X_1,X_2,X_3)+\R^3\times\{0_9\}=\R^6\times\{0_6\}.$$ 
As noticed in Remark~\ref{total} this proves the lemma.
\EOP


The controllability of system~\r{vz} follows from Proposition~\ref{Th1}: indeed, $\tilde A_{11}=\rho_1\Id$ is symmetric and invertible and hypothesis \r{rank} implies that $B_1=B_1 B_2^{-1}$  has a smaller set of eigenvectors than $J=\Id$.
Theorem~\ref{main} 
is therefore proven, because of Proposition~\ref{yup},   
under the stronger assumption that \r{rank} holds true.

\subsection{The proportional case}
In order to complete the proof of Theorem~\ref{main}, we have to prove that system \r{vz}--\r{R} is 
controllable when
\be\label{g_rank}
\rk B_2=3,\ \ \ \ \ B_1=\lb B_2,\ \ \ \ \ \lb\ne0.
\ee

Using the same normalization argument as in the previous section, we can assume that
$m=3$ and $B_2=\Id$. Moreover, we can take $\lb=1$ by performing the change of coordinates
$\xi'=\xi/\lb,\om'=\om,\zeta'=\zeta/\lb,R'=R$. Therefore, without loss of generality, 
$X_i={(e_i,e_i,\0_6)}$ for $i=1,2,3$.


Let us prove the Lie-bracket generating condition. 
We have
\[
[X_0,X_i] =(-\rho_1e_i,-\rho_2e_i,Re_i,RS(e_i))
\]
and 
\[
[[X_0,X_i],[X_0,X_j]] = (\0_6,2R(e_j\times e_i),RS(e_j\times e_i)).
\]
A further computation yields
\[
[[[X_0,X_i],[X_0,X_j]],[X_0,X_k]] =(\0_6,3R(e_k\times (e_j\times e_i)),
RS(e_k\times (e_j\times e_i))).
\]

Therefore, the Lie algebra $\L$ contains the following vector fields
\[
\begin{pmatrix}
e_i\\
e_i\\
\0_3\\
\0_3\\
\end{pmatrix}, \hspace{5em}
\begin{pmatrix}
-\rho_{1} e_j\\
-\rho_2e_j\\
Re_j\\
RS(e_j)
\end{pmatrix}, \hspace{5em}
\begin{pmatrix}
\0_3\\
\0_3\\
2Re_k\\
RS(e_k)
\end{pmatrix}, \hspace{5em}
\begin{pmatrix}
\0_3\\
\0_3\\
3Re_l\\
RS(e_l)
\end{pmatrix},
\]
for $i,j,k,l=1,2,3$.
As a consequence, $\L$ contains also 
\[
\begin{pmatrix}
(\rho_2-\rho_1)e_i\\
\0_3\\
Re_i\\
RS(e_i)\\
\end{pmatrix}, \hspace{5em}
\begin{pmatrix}
\0_3\\
(\rho_1-\rho_2)e_j\\
Re_j\\
RS(e_j)
\end{pmatrix}, \hspace{5em}
\begin{pmatrix}
\0_3\\
\0_3\\
Re_k\\
\0_3
\end{pmatrix}, \hspace{5em}
\begin{pmatrix}
\0_3\\
\0_3\\
\0_3\\
RS(e_l)
\end{pmatrix},
\]
for $i,j,k,l=1,2,3$. These vector fields form a moving frame on the manifold $\R^9\times\mathrm{SO}(3)$ (since $\rho_1\neq\rho_2$).

In order to conclude the proof of Theorem~\ref{main}
we need to show that \r{vz} is controllable.
Applying Lemma~\ref{propcontr3} this turns out to be equivalent to the controllability 
 of the following system, 
\begin{equation}
\label{syst17}
\dot x  =-\rho_1 x - v\times x + (\rho_2-\rho_1)v,\ \ \ \ x\in\R^3,
\end{equation}
with the control $v$ taking values in $\R^3$.
Decomposing the control as  $v=w x+W$ with $w\in\R$ and $W$ orthogonal to $x$,   equation~\r{syst17} rewrites as
\[
\dot x  =(-\rho_1+(\rho_2-\rho_1)w) x - W\times x.
\]
Clearly, choosing $W=\0_3$ allows to move as desired along $\spann(x)$, while the choice $w=\rho_1/(\rho_2-\rho_1)$ permits to attain the sphere of radius $\|x\|$ centered at the origin. 
The controllability of \r{vz} is thus proven and the proof of Theorem~\ref{main} completed.

\section{The case where the densities of
the microscopic organism and of the fluid have the same order of magnitude}\label{sme}

As pointed out in \cite[Section 2]{nancy},
the assumption
that the densities of
the microscopic organism and of the fluid have the same order of magnitude
leads to stronger simplifications
than in the general case studied in the previous sections. 
In particular, according 
to the remarks of page 6 the control  
system 
(2.18)--(2.27)
presented in \cite{nancy}
reduces in this case to
\br 
 -\Delta v+\nabla p=0&&\mbox{ on }F\times(0,T),\label{g1}\\
 \mathrm{div} v=0&& \mbox{ on }F\times(0,T),\\
 v(y,t)=\xi(t)+\om(t)\times y+\sum_{i=1}^m u_i(t)\psi_i(y)&& \mbox{ for }(y,t)\in\partial \Om\times(0,T),\\
\lim_{|y|\to\infty}v(y,t)=0&& \mbox{ for }t\in(0,T),\\
\int_{\partial \Om} \sigma(v,p)n\, ds=0&& \mbox{ on }(0,T),\label{int1}\\ 
\int_{\partial \Om}y\times \sigma(v,p)n\, ds=0&& \mbox{ on }(0,T),\label{int2}\\ 
\dot \zeta=R\xi&& \mbox{ on }(0,T),\\ 
\dot R=R S(\om)&& \mbox{ on }(0,T).
\label{g2}
\er
Physically, 
$v$ represents the field of velocities of 
the fluid in $F$ (the coordinates are attached to the body), 
$p$ is the pressure, $\xi$ and $\om$ are the linear and angular velocities of the organism and finally $\zeta$ and $R$  give its position and orientation with respect to a fixed frame. Recall that 
$\psi_1,\dots,\psi_m$ are the functions characterizing the control actions and that
$\sigma$ is the Cauchy stress  defined in \r{sigma}.

Given a  time-interval $[0,T]$ and a control $u\in L^\infty([0,T],\R^m)$
there exists a unique solution $(v,p,\xi,\om,\zeta,R)$ of \r{g1}--\r{g2} 
satisfying 
\brs
&v\in H^1([0, T], L^s (F)\cap D^{ 1,r}(F )\cap D^{2,\theta}(F )\cap \CC^\infty(F )),&\\
&\sup_{y\in F}(1+\|y\|)\|v(y,t)\|<\infty\ \ \ \mbox{for almost every }t\in[0,T],&\\
&p\in H^1([0, T], L^r (F)\cap D^{ 1,\theta}(F )\cap \CC^\infty(F )),&\\
&\xi\in H^1([0, T],\R^3 ),\ \omega\in H^1([0, T],\R^3),\ \zeta\in\CC^1([0, T ],\R^3),\ R\in\CC^1([0, T ], \SO(3)), &
\ers 
for $s\in(3,\infty]$, $r\in (3/2,\infty]$ and $\th\in(1,\infty)$
(see \cite{galdi,nancy}).

The 
finite-dimensional reduction of system \r{g1}--\r{g2}
can be obtained following the procedure proposed in
\cite{nancy}. 
To this extent,  as in 
Section~\ref{dynamics}, define the fundamental solutions $(h^{(i)},p^{(i)})$
and $(H^{(i)},P^{(i)})$ of the Stokes system and associate to them 
the matrices $\mau^k,\mad^k,\lau,\lad$.
Let $L$ be the $n\times m$ matrix 
$$L=-\lp\ba{cc}\mau^1&\mau^2\\ \mad^1&\mad^2\ea\rp^{-1}\lp\ba{c}\lau\\ \lad\ea\rp.$$

\begin{lemma}
For almost every $t\in (0,T)$
we have
$$\lp\ba{c} \xi(t)\\ \om(t)\ea\rp=L u(t).$$
\end{lemma}
\proof
According to \cite[Lemma 3.4]{nancy} 
we have 
\brs
\int_{\partial \Om} g^{(i)}\cdot v\,ds\ =\ \left[\int_{\partial \Om} \sigma(v,p)n \,ds\right]_i&=&0\\
\int_{\partial \Om} G^{(i)}\cdot  v\,ds\ =\ \left[\int_{\partial \Om} y\times \sigma(v,p)n \,ds\right]_i&=&0
\ers
where the equalities in the right-hand sides follow from 
\r{int1} and \r{int2}.

On the other hand, for almost every $t\in (0,T)$, 
$$
\int_{\partial \Om} g^{(i)}\cdot v\,ds=\int_{\partial \Om} g^{(i)}\cdot (\xi+\om\times y+\sum_{j=1}^m u_j \psi_j)\,ds=-(\mau^1 \xi+\mau^2 \om+\lau u)_i,
$$
and, similarly
$$
\int_{\partial \Om} G^{(i)}\cdot v\,ds=-(\mad^1 \xi+\mad^2 \om+\lad u)_i.
$$
Therefore,
$$
\lp\ba{c}\lau\\ \lad\ea\rp u=-\lp\ba{cc}\mau^1&\mau^2\\ \mad^1&\mad^2\ea\rp \lp\ba{c}\xi\\ \om \ea\rp
$$
which proves the lemma.
\EOP

\medskip

System \r{g1}--\r{g2} therefore reduces to the control system
\be
\label{syst-eq-dens}
\left\{
\ba{rl}
\dot{\zeta} & =R L^1 u,\\
\dot{R} & =RS(L^2 u),
\ea
\right.\ \ \ \ \ \ \ \ \ u\in \R^m,
\ee
where   $L^1,L^2\in \M_{3\times m}$ are such that
$$L=\lp\ba{c}L^1\\ L^2\ea\rp.$$

By homogeneity, since the set of admissible controls is the whole $\R^m$,
the 
controllability in arbitrary small time
of system \eqref{syst-eq-dens} 
is equivalent to the 
controllability of the control system whose admissible velocities
are given by the family of vector fields 
$\mathcal{F}=\{\pm X_{i}\mid 1\le i\le m\}$
defined by 
\be\label{Xi}
X_i(\zeta,R)={(R\: b_{i},R\: S(c_{i}))}
\ee 
where
the $b_{i}$'s 
denote the columns of $L^1$ and 
the $c_{i}$'s those of $L^2$. 
Moreover, since $\mathcal{F}$ is symmetrical,
the 
controllability of this family is equivalent to the fact
that the Lie algebra generated by $\mathcal{F}$, denoted by $\L$, 
is of rank 6 at every point
of $\R^{3}\times\mathrm{SO}(3)$ (see, eg, \cite[Corollary 5.11]{book2}).

\begin{thm}\label{same}
The control system \eqref{syst-eq-dens} is 
controllable in arbitrary small time if and
only if one of the following conditions is satisfied:
\begin{enumerate}
\item $\rank(L^2)\geq 2$ and $\T{(L^1)}L^2+\T{(L^2)}L^1\neq0_{m\times m}$; 
\item $\rank(L^2)=2$ and $\rank(L)\geq 3$.
\end{enumerate}
\end{thm}

\begin{proof} 
First of all notice that the Lie bracket between two vector fields 
of the form $Z_i(\zeta,R)={(R\: v_{i},R\: S(w_{i}))}$, $i=1,2$,
 is given by 
\be\label{bra}
[Z_1,Z_2](\zeta,R)=(R(w_{1}\times v_{2}-w_{2}\times v_{1}),R\, S(w_{1}\times w_{2})).
\ee

Let us prove the {}``if'' part of the theorem. Suppose that 
we can find in $\L$
two vector fields $Y_1,Y_2$ of the form
$$Y_{i}(\zeta,R)={(R\: \beta_{i},R\: S(\gamma_{i}))},\ \ \ \ i=1,2,$$ 
such that 
the vectors $\gamma_{1}$ and $\gamma_{2}$ 
are linearly independent and 
that
the inner product $\langle \beta_{1},\gamma_{1}\rangle$ is nonzero. 
We prove here below that, in this case, the Lie algebra $\L$
is of full rank at every point of $\R^{3}\times\mathrm{SO }(3)$.
Then we show that both $1$. and $2$. guarantee that, without loss of generality, 
such $Y_1$ and $Y_2$ can be found.

Let $Y_3=[Y_1,Y_2]$ whose expression can be obtained using \r{bra}.
Define the vector field $W_1$ as the bracket between $Y_3$ and $Y_1$. Applying again
\r{bra}  
we obtain that
 \[
W_1(\zeta,R)=[Y_3,Y_1](\zeta,R)=\begin{pmatrix}
R\left(2\langle \beta_{1},\gamma_{1}\rangle \gamma_{2}-(\langle \beta_{1},\gamma_{2}\rangle +\langle \beta_{2},\gamma_{1}\rangle )\gamma_{1}-\langle \gamma_{1},\gamma_{2}\rangle \beta_{1}+\| \gamma_{1}\|^{2}\beta_{2}\right)\\
R\, S(-\langle \gamma_{1},\gamma_{2}\rangle \gamma_{1}+\| \gamma_{1}\|^{2}\gamma_{2})
\end{pmatrix}.
\]
Let $W_2$ be the vector field obtained adding $\left\langle \gamma_{1},\gamma_{2}\right\rangle Y_{1}-\| \gamma_{1}\|^{2}Y_{2}$
to $W_1$, ie, 
\[
W_2(\zeta,R)=\begin{pmatrix}
R \left(2\langle \beta_{1},\gamma_{1}\rangle \gamma_{2}-\left(\langle \beta_{1},\gamma_{2}\rangle +\langle \beta_{2},\gamma_{1}\rangle \right)\gamma_{1} \right)\\
\0_3
\end{pmatrix}.
\]
Since $W_2$ is in $\L$, then the vector field
$Y_4=[W_2,Y_{1}]$ belongs to $\L$ as well and we have
that
$Y_4(\zeta,R)={\left(R(2\langle \beta_{1},\gamma_{1}\rangle \gamma_{1}\times \gamma_{2}),\0_3\right)}$. 
Consider now the vector fields 
$Y_5=[Y_{4},Y_{1}]$ and $Y_6=[Y_4,Y_{2}]$, both belonging to 
 $\L$, and whose expressions are
  $Y_5(\zeta,R)={\left(R(2\langle \beta_{1},\gamma_{1}\rangle \gamma_{1}\times(\gamma_{1}\times \gamma_{2})),\0_3\right)}$
and 
$Y_6(\zeta,R)={\left(R(2\left\langle \beta_{1},\gamma_{1}\right\rangle \gamma_{2}\times(\gamma_{1}\times \gamma_{2})),\0_3\right)}$. 
Since each matrix $R$ in $SO(3)$ is invertible, and because $\la \beta_1,\gamma_1\ra\ne0$, 
the vector fields $Y_1,\dots,Y_6$ span $T_{(\zeta,R)}\R^3\times SO(3)$ at every $(\zeta,R)\in \R^3\times SO(3)$ if and only if the matrix 
\[
\Delta=
\begin{pmatrix}
 \beta_{1}  &  \beta_{2}  &  \gamma_2\times \beta_1-\gamma_1\times \beta_2  &  \gamma_{1}\times \gamma_{2}  &  \gamma_{1}\times(\gamma_{1}\times \gamma_{2})  &  \gamma_{2}\times(\gamma_{1}\times \gamma_{2})\\
\gamma_{1}  &  \gamma_{2}  &  \gamma_{2}\times c_{1}  &  \0_3  & \0_3  &  \0_3
\end{pmatrix}
\]
is full-rank. 
The non-degeneracy of $\Delta$ follows easily 
from the assumption that the vectors $\gamma_{1}$ and $\gamma_{2}$ are linearly independent.

We shall prove now that the hypotheses $1$. and $2$. 
imply
the existence of two vector fields $Y_{1}$ and $Y_{2}$ 
as above. 

Let $r=\rank(L^2)$. 
Notice that for every matrix $\Gamma\in\GL(m)$ 
the 
reparameterization $u\to \Gamma u$ 
transforms the matrix $(L^1,L^2)$ into $(L^1\Gamma,L^2\Gamma)$.
The condition 
$\T{(L^1)}L^2+\T{(L^2)}L^1\neq0_{m\times m}$
is preserved by this transformation,
since, $\T{(L^1\Gamma)}L^2\Gamma+\T{(L^2\Gamma)}L^1\Gamma=\Gamma^T(\T{(L^1)}L^2+\T{(L^2)}L^1)\Gamma$.
Thus, without loss of generality,
we can assume that 
$L^2=(L^{2,1}\,|\0_{3\times(m-r)})$ where $L^{2,1}\in\M_{3\times r}$ is of rank $r$. 
Let us write $L^1$ as $(L^{1,1}\,|\,L^{1,2})$ where 
$L^{1,1}\in\M_{3\times r}$ and $L^{1,2}\in\M_{3\times (m-r)}$.

Suppose that $1$. holds.
If $\T{(L^{1,1})}L^{2,1}+\T{(L^{2,1})}L^{1,1}\neq\0_{r\times r}$, then we can find $1\leq j_1,j_2\leq r$ 
such that $\langle b_{j_1},c_{j_1}\rangle \neq0$, 
or $\langle b_{j_2},c_{j_2}\rangle \neq0$, or $ \langle b_{j_1},c_{j_2} \rangle + \langle b_{j_2},c_{j_1} \rangle \neq0$.
If $ \langle b_{j_1},c_{j_1} \rangle =\langle b_{j_2},c_{j_2} \rangle =0$, then the inner
product $ \langle b_{j_1}+b_{j_2},c_{j_1}+c_{j_2} \rangle = \langle b_{j_1},c_{j_2} \rangle + \langle b_{j_2},c_{j_1} \rangle $
is nonzero, and thus the vector fields $Y_1=X_{j_1}+X_{j_2}$ and $Y_2=X_{j_2}$ fulfill the required conditions. Otherwise,  $\T{(L^{2,1})}L^{1,2}\ne \0_{r\times (m-r)}$ and we can fix 
$1\leq j_1,j_2\leq r$ and $j_3>r$ such that $j_1\ne j_2$, $ \langle b_{j_1},c_{j_1} \rangle =0$
and $\langle b_{j_3},c_{j_1} \rangle\ne 0$. Then  $Y_1=X_{j_1}+X_{j_3}$ and $Y_2=X_{j_2}$ fulfill the required conditions.

Assume now that $2$. holds and $1$. does not. 
Then $c_{1}$ and $c_{2}$ are linearly independent and  
 $$ \langle b_{1},c_{1} \rangle= \langle b_{2},c_{2} \rangle= \langle b_{1},c_{2} \rangle + \langle b_{2},c_{1} \rangle =0.$$
Moreover, there exists $j>r=2$ such that $b_j\ne \0_3$ and $\langle b_{j},c_{1} \rangle= \langle b_{j},c_{2} \rangle=0$. Therefore,  there exists a real number $\alpha\neq0$
such that $b_{3}=\alpha(c_{1}\times c_{2})$. Then, according to \r{bra}, 
 $$[X_{1},X_{j}](\zeta,R)={\left(R(\alpha\, c_{1}\times(c_{1}\times c_{2})),\0_3\right)}.$$
As $\left\langle c_{2},c_{1}\times(c_{1}\times c_{2})\right\rangle \neq0$
the vector fields $Y_1=X_{1}$ and $Y_2=X_{2}+[X_{1},X_{j}]$ 
fulfill the required conditions. 

Let us prove now the ``only if'' part of the statement. 
First notice that if $r\le1$ then there
exists a vector $x_{0}\in\R^{3}\smallsetminus\{\0_3\}$ such that $L^2 u\times x_{0}=0$
for every control $u$. 
Multiplying equation
\eqref{syst-eq-dens} by $x_{0}$, we get that $\dot{R}x_{0}=0$ and so $R(t) x_{0}=R(0) x_{0}$ for every time $t$. Therefore, system \eqref{syst-eq-dens} is not controllable.
  
Assume now that $r=\rank(L)=2$ and $\T{(L^1)}L^2+\T{(L^2)}L^1=\0_{m\times m}$.
Let $Y=[X_1,X_2]$, whose expression, according to \r{bra}, is given by
$$Y(\zeta,R)={(R(c_{1}\times b_{2}-c_{2}\times b_{1}),{R\, S(c}_{1}\times c_{2}))}.$$
Since $\langle b_{i},c_{i} \rangle =0$ for $i=1,2$ and
$ \langle b_{1},c_{2} \rangle + \langle b_{2},c_{1} \rangle =0$,
we obtain that 
\brs
[Y,X_{1}]&=&- \langle c_{1},c_{2} \rangle X_{1}+\| c_{1}\|^{2}X_{2}\\
{[Y,X_{2}]}&=&-\| c_{2}\|^{2}X_{1}+ \langle c_{1},c_{2} \rangle X_{2}. 
\ers
Therefore, 
the Lie algebra $\L$, which is generated by $X_1$ and $X_2$,  is equal to
the linear space of vector fields spanned by $X_{1}$, $X_{2}$ and $Y$ and cannot be of full rank. 

Let now $r=3$ and 
$\T{(L^1)}L^2+\T{(L^2)}L^1=\0_{m\times m}$.
The condition $\T{(L^{2,1})}L^{1,2}= \0_{3\times (m-3)}$ implies that the columns of $L^{1,2}$ are orthogonal to all the elements of a basis of $\R^3$. Therefore, $L^{1,2}=\0_{3\times (m-3)}$. 
As for the columns of $L^{1,1}$, we easily obtain from $\T{(L^{1,1})}L^{2,1}+\T{(L^{2,1})}L^{1,1}=\0_{3\times 3}$
 that  
\begin{align*}
b_{1}= & -\alpha_{1 2}\,c_{2}-\alpha_{1 3}\,c_{3},\\
b_{2}= & \alpha_{1 2}\,c_{1}-\alpha_{2 3}\,c_{3},\\
b_{3}= & \alpha_{1 3}\,c_{1}+\alpha_{2 3}\,c_{2},
\end{align*}
for some $\al_{12},\al_{13},\al_{23}\in\R$. 
Without loss of generality we can assume that $(c_1,c_2,c_3)$ is a positively oriented orthonormal basis and an easy computation gives
 \begin{align*}
[X_{1},X_{2}]= & X_{3},\\
{}[X_{1},X_{3}]= & -X_{2},\\
{}[X_{2},X_{3}]= & X_{1},
\end{align*}
 which proves that $\L$ is equal to the 
linear space of vector fields spanned by 
$X_{1},$ $X_2$ and $X_3$ and cannot be of full rank.
\end{proof}


\bibliographystyle{plain}
\bibliography{biblio}

\end{document}